\documentclass[11pt]{amsart}

\usepackage{epigamath}

\usepackage[english]{babel}

\numberwithin{equation}{section}

\usepackage[shortlabels]{enumitem}

\usepackage{amsmath}
\usepackage{amssymb}
\usepackage{graphicx}
\usepackage{multirow}
\DeclareGraphicsExtensions{.png,.pdf,.eps}
\usepackage[all,2cell]{xy}
\usepackage{tikz}
\usepackage{pgf}
\usepackage{mathdots}
\usepackage{hf-tikz}

\usetikzlibrary{cd}
\usepackage{tikz-cd}

\usepackage{array}
\newcolumntype{L}{>{$}l<{$}} 
\newcolumntype{C}{>{$}c<{$}} 

\newtheorem{theorem}{Theorem}[section]
\newtheorem{lemma}[theorem]{Lemma}
\newtheorem{corollary}[theorem]{Corollary}
\newtheorem{proposition}[theorem]{Proposition}

\theoremstyle{definition}

\newtheorem{definition}[theorem]{Definition}
\newtheorem{notation}[theorem]{Notation}
\newtheorem{setup}[theorem]{Setup}
\newtheorem{construction}[theorem]{Construction}

\theoremstyle{remark}
\newtheorem{remark}[theorem]{Remark}
\newtheorem{example}[theorem]{Example}

\DeclareMathOperator{\Spec}{Spec}

\DeclareMathOperator{\HH}{H}

\def\sing{\operatorname{sing}}

\newcommand\PP{{\mathbb{P}}}

\newcommand\FF{{\mathbb{F}}}

\def\C{{\mathbb C}}

\def\G{{\mathbb G}}

\def\P{{\mathbb P}}
\def\Q{{\mathbb Q}}

\def\Z{{\mathbb Z}}

\def\cA{{\mathcal A}}

\def\cD{{\mathcal D}}

\def\cH{{\mathcal{H}}}
\def\cI{{\mathcal I}}

\def\cL{{\mathcal L}}
\def\cM{{\mathcal M}}
\def\cN{{\mathcal{N}}}
\def\cO{{\mathcal{O}}}

\def\cU{{\mathcal U}}

\def\cY{{\mathcal Y}}
\def\Q{{\mathbb{Q}}}
\def\G{{\mathbb{G}}}

\def\cExt{{\mathcal E}xt}

\def\operatorname#1{\mathop{\rm #1}\nolimits}

\def\DA{{\rm A}}
\def\DB{{\rm B}}
\def\DC{{\rm C}}
\def\DD{{\rm D}}

\def\Proj{\operatorname{Proj}}

\def\Chow{\operatorname{Chow}}
\def\Ext{\operatorname{Ext}}
\def\Exc{\operatorname{Exc}}
\def\Hilb{\operatorname{Hilb}}
\def\Hom{\operatorname{Hom}}

\def\Pic{\operatorname{Pic}}
\def\Hom{\operatorname{Hom}}

\def\Spec{\operatorname{Spec}}

\def\codim{\operatorname{codim}}

\def\rk{\operatorname{rk}}

\def\NE{{\operatorname{NE}}}

\def\Nef{{\operatorname{Nef}}}

\def\Nu{{\operatorname{N_1}}}
\def\NU{{\operatorname{N^1}}}

\def\PGL{\operatorname{PGL}}

\def\prul{s}
\def\prur{d}

\def\GX{\mathcal{G}\!X\hspace{-0.5pt}}
\def\CX{\mathcal{C}\!X\hspace{-0.5pt}}

\newcommand{\pb}{\ar@{}[dr]|{\text{\pigpenfont J}}}
\def\ol{\overline}

\newcommand{\shse}[3]{0 ~\lra ~#1~ \lra ~#2~ \lra ~#3~ \lra~ 0}
\makeatletter
\newcommand{\xleftrightarrow}[2][]{\ext@arrow 3359\leftrightarrowfill@{#1}{#2}}
\makeatother
\newcommand{\xdasharrow}[2][->]{
\tikz[baseline=-\the\dimexpr\fontdimen22\textfont2\relax]{
\node[anchor=south,font=\scriptsize, inner ysep=1.5pt,outer xsep=2.2pt](x){#2};
\draw[shorten <=3.4pt,shorten >=3.4pt,dashed,#1](x.south west)--(x.south east);
}}

\newcommand\m{{\mathfrak m}}

\newcommand\lra{\longrightarrow}

\def\Mo{\operatorname{\hspace{0cm}M}}

\newcommand{\supth}[1]{\ensuremath{#1^{\mathrm{th}}}}

\EpigaVolumeYear{10}{2026} \EpigaArticleNr{16} \ReceivedOn{April 18, 2025}
\InFinalFormOn{February 19, 2026}
\AcceptedOn{March 13, 2026}

\title{Chow quotients of $\boldsymbol{\C^*}$-actions on convex varieties}
\titlemark{Chow quotients of $\boldsymbol{\C^*}$-actions on convex varieties}

\author{Gianluca Occhetta}
\address{Dipartimento di Matematica, Universit\`a degli Studi di Trento, via
Sommarive 14 I-38123 Povo di Trento (TN), Italy}
\email{gianluca.occhetta@unitn.it}
\author{Luis E. Sol\'a Conde}
\address{Dipartimento di Matematica, Universit\`a degli Studi di Trento, via
Sommarive 14 I-38123 Povo di Trento (TN), Italy}
\email{eduardo.solaconde@unitn.it}

\authormark{G.~Occhetta and L.\,E.~Sol\'a Conde}

\AbstractInEnglish{In this paper, we study the Chow quotient $\CX$ of a convex variety $X$ of Picard number~$1$ by the action of a $1$-dimensional torus with no nontrivial finite isotropy. Examples of these actions can be found in the rational homogeneous framework. We prove that the subvariety of $\CX$ parametrizing reducible torus-invariant cycles is a simple normal crossing divisor, and we compute the nef and Mori cones of $\CX$, and its anticanonical divisor.}

  

\acknowledgement{The authors were partially supported by INdAM--GNSAGA.}

\begin{document}



\maketitle

\begin{prelims}

\DisplayAbstractInEnglish

\bigskip

\DisplayKeyWords

\medskip

\DisplayMSCclass

\end{prelims}


\newpage

\setcounter{tocdepth}{1}

\tableofcontents



\section{Introduction}\label{sec:intro}

Starting with Mumford's celebrated geometric invariant theory (GIT) in the 1960s (\textit{cf.}~\cite{MFK}), quotients have become an essential technique in modern algebraic geometry. They allow one to construct new examples of algebraic varieties, in particular moduli spaces and parameter spaces of different types.

In Mumford's theory, quotients are constructed upon  polarized varieties by means of stability conditions. While variations of these stability conditions can be interpreted in birational-geometric terms, giving rise to the concept of Mori dream space (see \cite{Thaddeus1996,HuKeel}), GIT also motivated some geometers to define intrinsic concepts of quotients by the action of an algebraic group.

One of these is the concept of Chow quotient, introduced by  
Kapranov (\textit{cf.}~\cite{Kap}). In a nutshell, the Chow quotient of a projective variety $X$ by the action of an algebraic group $H$  is the closure of the subvariety of the Chow scheme $\Chow(X)$ parametrizing the closure of the general orbit in $X$.

Remarkably, the case in which $X=G/P$ is a complex projective rational homogeneous variety and $H$ is a torus leads to a number of interesting Chow quotients. A noteworthy example of this kind is the Grothendieck--Knudsen moduli space $\ol{M}_{0,n}$ of stable punctured curves of genus zero: Kapranov showed in 1993 (see \cite[Chapter~4]{Kap}) that it is the Chow quotient of the Grassmannian of lines in $\P^{n-1}$ by the action of a maximal torus in $\PGL(n)$. 

Furthermore, as Thaddeus noted in 2001 (see \cite{Thaddeus}), some specially interesting varieties can be obtained as Chow quotients of  rational homogeneous spaces by the actions of $1$-dimensional complex tori. More precisely, the Chow quotients of
\begin{itemize}[leftmargin=25pt]
\item Grassmannians of $(n-1)$-dimensional linear subspaces of $\P^{n+m-1}$ ($n \le m$),
\item Lagrangian Grassmannians of $(n-1)$-dimensional isotropic subspaces in $\P^{2n-1}$ (with respect to a fixed nondegenerate skew-symmetric form),
\item orthogonal Grassmannians of $(n-1)$-dimensional linear subspaces in a quadric of dimension $2n-2$ 
\end{itemize}
by some particular $\C^*$-actions are isomorphic, respectively, to the moduli spaces
\begin{itemize}[leftmargin=25pt]
\item   of complete collineations from a vector space of dimension $m$ to a vector space of dimension $n$, 
\item of complete quadrics of dimension $n-1$,  
\item  of complete skew-forms on a vector space of dimension $n$. 
\end{itemize}

These three moduli spaces have been extensively studied in the literature (see, for instance, \cite{Lak82,Lak87,Lak88, Thaddeus, Vai82,Vai84}), due to their importance in the framework of enumerative geometry. In particular, some of their birational-geometric properties (such as their nef, movable, and effective cones) have been studied in \cite{Mas1,Mas2}. Many of these results are based on classical constructions of these varieties, by means of sequences of blowups of a projective space.  We believe that their description as Chow quotients may help understand the geometry of these varieties in a unified broader setting, and that some of their geometric properties will be shared by other Chow quotients of rational homogeneous varieties.

The setting in which our arguments work is the class of convex varieties, which are smooth projective varieties whose tangent bundle is nef on rational curves. Note that this class contains the class of Fano manifolds whose tangent bundle is nef, and that it has been conjectured by Campana and Peternell that rational homogeneous varieties are the only Fano manifolds with this property (\textit{cf.}~\cite{CP,MOSWW}); to the best of our knowledge, no examples of nonhomogeneous uniruled convex varieties are known.

We assume that the Picard number of the variety $X$ whose quotient we consider is $1$, and that the action of $\C^*$ is {\em equalized}; this means that proper isotropy subgroups of the action are trivial (see Section~\ref{sssec:equal}). These assumptions are satisfied by the actions leading to complete collineations, quadrics, and skew-forms, but also by many other examples of torus actions on rational homogeneous varieties whose Chow quotients are, to the best of our knowledge, still unknown.

We will use some background results, which we have presented in \cite{WORS6,OS2}, and which can be applied to a variety $X$ endowed with a $\C^*$-action as above. In particular, we will use that, in our setting, the Chow quotient $\CX$ is smooth and that the critical values of the action---with respect to the ample generator $L$ of the Picard group of  the variety---are of the form $0,d,2d,\dots,rd$ for some positive integers $d,r$; the number $r$ is called the {\em criticality} of the action.
Furthermore, there exists  exactly one fixed-point component $Y_j$ of $L$-weight $jd$ for every $j$ (see Section~\ref{ssec:polar} and Proposition~\ref{prop:onecomp}).

A finer study of deformations of $\C^*$-invariant rational $1$-cycles allows us to describe the Chow quotient $\CX$ of $X$ as the union of the set of irreducible deformations of the general $1$-dimensional $\C^*$-invariant rational curve and a divisor (called the {\em boundary} of $\CX$) with simple normal crossings (see Theorem~\ref{thm:smoothborder} and Remark~\ref{rem:main1}). 

\begin{theorem}\label{thm:main1}
  Let $X$ be a convex variety of Picard number $1$,   endowed with a nontrivial equalized $\C^*$-action of criticality $r$. Then the Chow quotient $\CX$ of\, $X$ by the action of\, $\C^*$ is smooth, and the subscheme of\, $\CX$ parametrizing reducible $1$-cycles is a divisor with simple normal crossings and $r-1$ irreducible components.
\end{theorem}

In a nutshell, the convexity assumption allows us to show that every deformation of a $1$-cycle parametrized by an element of $\CX$ can be realized as a $\C^*$-invariant deformation of the $1$-cycle in $X$ (see Section~\ref{ssec:boundary}). 
On the other hand, we know that, in our setting, the Chow quotient $\CX$ is the normalization of the inverse limit of the GIT quotients of $X$. In particular, we get  information about some contractions of $\CX$; we show that we can use these contractions to describe completely the Mori cone of $\CX$. More precisely,  denoting by $\GX_{j,j}$, $j=0,\dots,r$, the semigeometric quotients of $X$ (see Section~\ref{sssec:quotients}) and  by $\cL_{j,j}$ the pullback to $\CX$ of an ample line bundle on $\GX_{j,j}$,  we  prove the following (see Theorem~\ref{thm:nefcone} for the precise statement). 

\begin{theorem}\label{thm:main2}
Let $X$ be a convex variety of Picard number $1$ endowed with a nontrivial equalized $\C^*$-action of criticality $r$. Then the nef cone of\, $\CX$ is generated by the classes $[\cL_{i,i}]$, $i=0,\dots,r$. 
Moreover, the Picard number of $\CX$ is equal to 
\[
\rho(\CX)=r+1-\sharp\left\{i\in\{0,r\}\mid \dim(Y_i)=0\right\}=r-1,r,\mbox{ or }r+1.
\]
\end{theorem}

For the second part of the statement, we note that $\cL_{0,0}$ and $\cL_{r,r}$ are trivial if and only if $\GX_{0,0}\simeq Y_0$ and $\GX_{r,r}\simeq Y_r$ are isolated points.

Our methods allow us to describe the Mori cone of $\CX$ as well. Its extremal rays are generated by the classes of some curves in $\CX$ that we call {\em primal}; these curves admit a very natural description as families of $\C^*$-invariant cycles (see Section~\ref{ssec:primal}).

Finally, our results allow us to compute the anticanonical divisor of $\CX$ in terms of some invariants of the $\C^*$-action. We then present  examples of actions in which $-K_{\CX}$ is ample (\textit{i.e.}, $\CX$ is a Fano manifold), nef but not ample, and not nef. 

A key point in our arguments is the following: The maps $\CX\to \GX_{i,i}$ factor through the Chow quotients $\CX_{i,j}$ of certain birational modifications $X_{i,j}$ of $X$ (for $i <j\in\{0,\dots, r\}$), called {\it prunings} (see Section~\ref{sssec:quotients}). The varieties $\CX_{i,j}$, called partial Chow quotients of $X$, fit in a diagram of normalized Cartesian squares:
\[
\begin{tikzcd}[
  column sep={2.7em,between origins},
  row sep={2.7em,between origins}]
&&&&&&\CX\arrow[rd] \arrow[dl] &&&&&&\\
&&&&&\CX_{0,r-1}\arrow[rd] \arrow[dl]&&\CX_{1,r}\arrow[rd] \arrow[dl]&&&&&\\
&&&&\CX_{0,r-2} \arrow[dl]\arrow[rd]&&\CX_{1,r-1}\arrow[rd] \arrow[dl]&&\CX_{2,r}\arrow[dl]\arrow[rd]&&&&\\
&&& \arrow[ld,dotted,no head]&& \arrow[ld,dotted,no head]&&\arrow[rd,dotted,no head]&&\arrow[rd,dotted,no head]&&&\\
&&\CX_{0,2}\arrow[rd] \arrow[dl]&&\CX_{1,3} \arrow[dl]&&\dots&&\CX_{r-3,r-1}\arrow[rd] &&\CX_{r-2,r}\arrow[rd] \arrow[dl]&&\\
&\CX_{0,1}
&&\CX_{1,2}
&&\dots&&\dots&&\CX_{r-2,r-1}
&&\CX_{r-1,r}\rlap{.}
&
\end{tikzcd}
\]
This idea has already allowed us to prove that, under the assumptions of Theorem~\ref{thm:main1},  $\CX_{i,j}$ is smooth if $i=0$ or  $j=r$, and that the maps 
\[\xymatrix@C=17pt{\GX_{0,1}&\CX_{0,2}\ar[l]&\ \cdots\ \ar[l]&\CX\ar[l]\ar[r]&\ \cdots\ \ar[r]&\CX_{r-2,r}
\ar[r]&\GX_{r-1,r}}\]
are blowups with smooth centers (see \cite[Propositions~6.4 and 6.5]{WORS6}).  Here we extend these result and show that, under the hypotheses of Theorem~\ref{thm:main1}, all the maps between the partial Chow quotients are blowups with smooth centers, and that all the partial Chow quotients $\CX_{i,j}$ are smooth (see Theorem~\ref{thm:smoothpart}).

\subsection*{Acknowledgments}
The authors would like to thank Jarek Wi\'sniewski for inspiring discussions about Chow quotients.


\section{Background material}

Throughout the paper, we will work over the field of complex numbers. The term {\em rational homogeneous variety} will always refers to a projective variety that supports a transitive algebraic action of a semisimple group. When working with these varieties, we will use the notation of marked Dynkin diagrams, as described, for instance, in \cite[Section~3.1]{WORS5}.

\subsection{$\boldsymbol{\C^*}$-actions}\label{sec:c*}

In this section, we  present some background material on $\C^*$-actions. We refer to \cite{BB,CARRELL,WORS1} for further details. In what follows, $X$ will be a smooth projective variety, endowed with a nontrivial $\C^*$-action.

We denote by $X^{\C^*}$  the fixed locus of the action, and by $\cY$ the set of irreducible fixed-point components: 
$$X^{\C^*}=\bigsqcup_{Y\in \cY}Y.$$ 
 The {\em sink} and the {\em source} of the action, often called the {\em extremal fixed-point components}, are the fixed-point components containing, respectively, the limiting points 
$$
x_-:=\lim_{t\to 0}t^{-1}x, \quad x_+:=\lim_{t\to 0}tx
$$
of a general $x \in X$.  The fixed-point components different from the sink and the source are called {\em inner}.  Every $Y\in\cY$ is smooth (\textit{cf.}~\cite{IVERSEN}), and the normal bundle of $Y$ in $X$ splits into two subbundles, on which $\C^*$ acts with positive and negative weights, respectively:
\begin{equation}
N_{Y|X}\simeq N^+(Y)\oplus N^-(Y).
\label{eq:normal+-}
\end{equation}
We set $\nu^{\pm}(Y):=\rk (N^{\pm}(Y))$.

\subsubsection{Bia{\l}ynicki-Birula decomposition}
Given any set $S \subset X$, we define
\begin{equation}\label{eq:BBcells}
X^\pm(S):=\{x\in X\mid x_\pm\in S\}, \quad
B^\pm(S):=\overline{X^\pm(S)}.
\end{equation}
When $S=Y$ with $Y\in \cY$, we call $X^\pm(Y)$ the positive and the negative {\it Bia{\l}ynicki-Birula cells} of the action at~$Y$  (BB-cells, for short);  we have decompositions 
$$X=\bigsqcup_{Y\in\cY}X^+(Y)=\bigsqcup_{Y\in\cY}X^-(Y).$$  
We will denote by $B^\pm(Y)$ the closures of the BB-cells $X^\pm(Y)$. Every cell $X^\pm(Y)$ has a natural morphism onto~$Y$, sending every $x\in X^\pm(Y)$ to its limiting point $x_\pm$. It was shown by Bia{\l}ynicki-Birula that $X^\pm(Y)$ is an affine bundle, locally isomorphic to $N^\pm(Y)$.

\subsubsection{Actions on polarized pairs}\label{ssec:polar} Given  a line bundle $L\in \Pic(X)$, there exists a linearization of the $\C^*$-action on it, so that for every $Y\in \cY$, $\C^*$ acts on $L_{|Y}$ by multiplication with a character $m\in \Mo(\C^*)=\Hom(\C^*,\C^*)$, called the {\em weight of the linearization on $Y$}. Two linearizations differ by a character of $\C^*$. In particular, for any  $L \in \Pic(X)$, there exists a unique linearization (called {\em normalized linearization})  whose weight at the sink is equal to zero. 
Fix an isomorphism $\Mo(\C^*)\simeq \Z$; then  this linearization defines a map  $\mu_L\colon\cY\to\Z$, called the {\em weight map}.

A {\em $\C^*$-action on the polarized pair $(X,L)$} is a $\C^*$-action on $X$, together with the weight map $\mu_L$ determined by an ample line bundle $L$.  In this case, we denote by
$$
0=a_0<\dots<a_r
$$ 
the weights $\mu_L(Y)$, $Y\in\cY$, 
and set 
$$
Y_i:=\bigcup_{\mu_L(Y)=a_i}Y.
$$ 
The values $a_i$ will be called the {\em critical values}, and the number $r$ will be called the {\em criticality} of the action.
The minimum and maximum of the critical values are achieved only at the  sink and the  source, 
which are then equal to $Y_0$ and $Y_r$, respectively. The value $\delta=a_r$ is called the {\em bandwidth} of the $\C^*$-action on $(X,L)$. 

\subsubsection{Equalized actions}\label{sssec:equal}

A $\C^*$-action  is {\em  equalized} at $Y\in\cY$ if for every $x\in (X^{-}(Y)\cup X^{+}(Y))\setminus Y$, the isotropy group of the action at $x$ is trivial. Equivalently (see \cite[Lemma 2.1]{WORS3}), $\C^*$ acts on $N^+(Y)$ with all the weights equal to $+1$ and on $N^-(Y)$ with all the weights equal to $-1$. The action is called {\em equalized} if it is equalized at every $Y \in \cY$.

For an equalized action, the closure of any $1$-dim\-ens\-ional orbit is a smooth rational curve, whose $L$-degree may be computed in terms of the weights at its extremal points (\textit{cf.} \cite[Lemma 2.2]{WORS3}, \cite[Corollary 3.2]{RW}).  

\begin{lemma}[AM vs. FM]\label{lem:AMvsFM}
Let $(X,L)$ be a polarized pair with an equalized $\C^*$-action, and let $C$ be the closure of a $1$-dimensional orbit, with sink  $x_-$ and source $x_+$. Then $C$ is a smooth rational curve of\, $L$-degree  equal to $\mu_L(x_+)-\mu_L(x_-)$.
\end{lemma} 

Using Lemma~\ref{lem:AMvsFM}, one can show (see \cite[Corollary 2.4]{OS2}) the following. 

\begin{corollary} \label{cor:can} 
Let $X$ be a smooth projective variety  with an equalized $\C^*$-action, and let $C$ be the closure of a $1$-dimensional orbit, with sink  in $Y$ and source in  $Y'$. Then 
\[
-K_X \cdot C =\left(\nu^+(Y')-\nu^-(Y')\right)-\left(\nu^+(Y) -\nu^-(Y)\right).
\]
\end{corollary}

\subsection{Quotients}\label{sssec:quotients} For  more details on the contents of this section,  we refer the reader to \cite{WORS3} and \cite{WORS6}. A $\C^*$-action on a polarized variety $(X,L)$ allows one to define projective quotients of $X$, as follows. Set, for every $\tau \in [0,\delta]\cap \Q$, 
 \[A(\tau):=\bigoplus_{\substack{m\geq 0\\m\tau\in\Z}}\HH^0(X,mL)_{m\tau},\quad \GX(\tau):=\Proj(A(\tau)).
\]
By \cite[Section 2.2]{WORS6}, $\GX(\tau)$ is the same for any  $\tau\in (a_i,a_{i+1})\cap \Q$, so we set 
\begin{itemize}[itemsep=5pt,topsep=5pt]
\item $\GX_{i,i}:=\GX(\tau)$ for $\tau=a_i$,  
\item $\GX_{i,i+1}:=\GX(\tau)$ for $\tau\in(a_i,a_{i+1})\cap \Q$.
\end{itemize}
Note that $\GX_{i,i+1}$ is a {\em geometric quotient} of $X$ for every $i$, while $\GX_{i,i}$ is only a good categorical quotient (see \cite[Section 6.1, Definition, p.~94]{DolgachevGIT}). Following \cite[Chapter 3, p.~10]{BBC}, we say that $\GX_{i,i}$ is {\em semigeometric}.  
Furthermore, the varieties $\GX_{i,i+1}$ are birationally equivalent, and the natural maps among them fit in the following commutative diagram, where the diagonal arrows are contractions and the horizontal maps are Atiyah flips (in the sense of \cite[Definition 6.1]{WORS1}): 
\begin{equation}
\begin{tikzcd}[
  column sep={3.4em,between origins},
  row sep={2.7em,between origins},
]
&\GX_{0,1} \arrow[rr,dashed] \arrow[rd] \arrow[dl]&&\GX_{1,2} \arrow[dl]&\dots\ \ &\GX_{r-2,r-1} \arrow[rr,dashed] \arrow[rd] &&\GX_{r-1,r}\arrow[rd] \arrow[dl]&\\
\GX_{0,0}&&\GX_{1,1}&\dots&&\dots\ \ &\GX_{r-1,r-1}&&\GX_{r,r}\rlap{.}
\end{tikzcd}
\label{eq:GITquot}
\end{equation}

\subsubsection{Prunings} \label{ssec:prun}

Given an equalized $\C^*$-action of criticality $r$ and bandwidth $\delta$ on the smooth polarized pair $(X,L)$,  we consider the blowup $\beta\colon X^\flat\to X$ of $X$ along the sink $Y_0$  and the source $Y_r$, with exceptional divisors $Y^\flat_0,Y^\flat_r$, and the induced $\C^*$-action.
Given rational numbers $\tau_-<\tau_+$ in $[0,\delta]$, we set
\[
L(\tau_-,\tau_+):=\beta^*L-\tau_-Y^\flat_0-(\delta-\tau_+)Y^\flat_r 
\]
and define the {\em pruning} of $(X,L)$ with respect to $\tau_-,\tau_+$ as
\[X(\tau_-,\tau_+):=\Proj \bigoplus_{m\in k\Z_{\geq 0}}\HH^0(X^\flat,mL(\tau_-,\tau_+)),\]
where $k>0$ is the minimum integer such that $k\tau_\pm\in \Z_{> 0}$.
As shown in \cite[Lemma 3.5]{WORS6}, for values $\tau_-,\tau_+\in([0,\delta]\cap\Q)\setminus\{a_0,\dots a_r\}$, the pruning $X(\tau_-,\tau_+)$ depends only on the intervals $(a_i,a_{i+1})$, $ (a_{j-1},a_j)$ containing $\tau_-,\tau_+$, respectively. Hence we can set
$$X_{i,j}:=X(\tau_-,\tau_+)\quad\mbox{whenever } \tau_-\in(a_i,a_{i+1}), \;\tau_+\in (a_{j-1},a_j).$$ 

\subsubsection{Chow quotient}

Given a projective variety $X$ with a nontrivial $\C^*$-action, the {\em normalized Chow quotient} of the action is the normalization $\CX$ of the closure of the subscheme of $\Chow(X)$ parametrizing the closures of general orbits.

More precisely, one may find a nonempty open set $V\subset X$ of points $x$ whose orbits $\overline{\C^*\cdot x}$ belong to a single homology class in $X$. By shrinking $V$, if necessary, the geometric quotient of $V$ by the action of $\C^*$ exists. Taking closures in $X$ of the orbits of points of $V$, we obtain an algebraic irreducible family of rational curves in $X$ and a morphism $\psi\colon V/\C^*\to \Chow(X)$.

The {\em Chow quotient} of $X$ by the action of $\C^*$ is then defined as the closure 
$$
\overline{\CX}:=\overline{\psi(V/\C^*)}\subset \Chow(X).
$$
Its normalization $\CX$ will be called the {\em normalized Chow quotient} of $X$ by the $\C^*$-action.  Pulling back the universal family of $\Chow(X)$ to $\CX$, we obtain a family of $1$-dimensional varieties $p\colon\cU\to \CX$, with evaluation morphism $q\colon\cU\to X$.

\begin{remark}\label{rem:chains} 
Let us consider a smooth polarized pair $(X,L)$ endowed with an equalized $\C^*$-action of criticality~$r$, with critical values $a_0,\dots,a_r$. By \cite[Lemma~4.3]{WORS6}, the fibers of $p$ are reduced $1$-dimensional schemes with smooth irreducible components, whose dual graphs are of type $\DA$. Thus any element of the family parametrized by $\CX$ takes the form $C=C_1\cup\dots\cup C_k$, 
where 
\begin{enumerate}[label=(C\arabic*),ref=C\arabic*]
\item\label{C1} every $C_j$ is the (smooth) closure of a $1$-dimensional orbit; 
\item\label{C2} the singular points of $C$ are the (transversal) intersections $P_j:=C_j\cap C_{j+1}$ and are fixed points of weights $a_{i_j}$, with  $0<i_1<\dots<i_{k-1}<r$; 
\item\label{C3} every $P_j$ is the source of $C_j$ and the sink of $C_{j+1}$, for $j=1,\dots,k-1$;  
\item\label{C4} the sink of $C_1$ and the source of $C_k$ belong to $Y_0$ and $Y_r$, respectively.
\end{enumerate}
Note that property~(\ref{C4}) follows from Lemma~\ref{lem:AMvsFM} together with the fact that the $L$-degree of $C$ equals the $L$-degree of the general element of $\CX$, which is $\delta=a_r-a_0$.
\end{remark}

\begin{remark}\label{rem:Hilb} From this description of the fibers of $p$, one may prove (see \cite[Lemma~4.5]{WORS6}) that the family $p$ is flat. In particular, $\CX$ is isomorphic to the normalization of an irreducible $\C^*$-invariant subvariety of $\Hilb(X)$, which we call the {\em Hilbert quotient} of the action (see \cite[Remark~4.6]{WORS6}). 
\end{remark}

\textit{A priori}, it is not clear that any $\C^*$-invariant subscheme $C=C_1\cup\dots\cup C_k$ satisfying the properties (\ref{C1})--(\ref{C4}) listed in Remark~\ref{rem:chains} belongs to $\CX$. However, this has been proved to be true, under the following mild additional assumptions. 

\begin{setup}\label{set:WORS6}
The pair $(X,L)$ is a smooth polarized pair endowed with an equalized $\C^*$-action of criticality $r$, with critical values $0=a_0<\dots<a_r$, such that $\nu^\pm(Y) \ge 2$ for every fixed-point component of weight $a_i$, with $2 \le i \le r-2$.
\end{setup} 

\begin{remark}\label{rem:picone}
The technical assumption above on the $\nu^\pm(Y)$ is fulfilled if the variety $X$ has Picard number~$1$ (\textit{cf.}~\cite[Lemma 2.8]{WORS1} and \cite[Lemma 4.5]{WORS3}).
\end{remark}

\begin{corollary}[\textit{cf.} {\cite[Corollary~5.9]{WORS6}}]\label{cor:CXcontainsall}
Let $(X,L)$ be as in Setup~\ref{set:WORS6}. Then every $\C^*$-invariant $1$\nobreakdash-di\-men\-sional reduced subscheme $C$ of\, $X$ with irreducible components $C_1,\dots,C_k$ satisfying the properties {\rm(\ref{C1})--(\ref{C4})}  listed in Remark~\ref{rem:chains}, is an element of the family parametrized by $\CX$.
\end{corollary}

The proof  uses the recursive description of $\CX$ given in \cite[Theorem 1.1 and Corollary 7.1]{WORS6}, which we now recall. 

\begin{theorem}\label{thm:rhombuses}
Let $(X,L)$ be as in Setup~\ref{set:WORS6}. Denote by $X_{i,j}$, $0\leq i<j\leq r$, the corresponding prunings of\, $X$, and by $\CX_{i,j}$ their normalized Chow quotients. Then $\CX=\CX_{0,r}$, and we have a commutative diagram 
\begin{equation}\label{fig:Chow0}
\begin{tikzcd}[
  column sep={2.45em,between origins},
  row sep={2.7em,between origins}]
&&&&&&\CX_{0,r}\arrow[rd,"\prur"] \arrow[dl,"\prul",labels=above left] &&&&&&\\
&&&&&\CX_{0,r-1}\arrow[rd,"\prur"] \arrow[dl,"\prul",labels=above left]&&\CX_{1,r}\arrow[rd,"\prur"] \arrow[dl,"\prul",labels=above left]&&&&&\\
&&&&\CX_{0,r-2} \arrow[dl,"\prul",labels=above left]\arrow[rd,"\prur"]&&\CX_{1,r-1}\arrow[rd,"\prur"] \arrow[dl,"\prul",labels=above left]&&\CX_{2,r}\arrow[dl,"\prul",labels=above left]\arrow[rd,"\prur"]&&&&\\
&&& \arrow[ld,dotted,no head]&& \arrow[ld,dotted,no head]&&\arrow[rd,dotted,no head]&&\arrow[rd,dotted,no head]&&&\\
&&\CX_{0,2}\arrow[rd,"\prur"] \arrow[dl,"\prul",labels=above left]&&\CX_{1,3} \arrow[dl,"\prul",labels=above left]&&\dots&&\CX_{r-3,r-1}\arrow[rd,"\prur"] &&\CX_{r-2,r}\arrow[rd,"\prur"] \arrow[dl,"\prul",labels=above left]&&\\
&\GX_{0,1}\arrow[ld]\arrow[rd]&&\GX_{1,2}\arrow[ld]\arrow[rd]&&\dots&&\dots&&\GX_{r-2,r-1}\arrow[ld]\arrow[rd]&&\GX_{r-1,r}\arrow[ld]\arrow[rd]&\\
\GX_{0,0}&&\GX_{1,1}&&\GX_{2,2}&&\dots&&\GX_{r-2,r-2}&&\GX_{r-1,r-1}&&\GX_{r,r}\rlap{,}
\end{tikzcd}
\end{equation}
where all the arrows denoted by $\prul,\prur$ are blowups and every rhombus in the diagram is a normalized Cartesian square.
\end{theorem}

In categorical language, we may simply say that $\CX$ is the normalized inverse limit of the GIT quotients
\[\begin{tikzcd}[
  column sep={2.5em,between origins},
  row sep={2.9em,between origins}]
&\GX_{0,1}\arrow[ld]\arrow[rd]&&\GX_{1,2}\arrow[ld]\arrow[rd]&&\dots&&\dots&&\GX_{r-2,r-1}\arrow[ld]\arrow[rd]&&\GX_{r-1,r}\arrow[ld]\arrow[rd]&\\
\GX_{0,0}&&\GX_{1,1}&&\GX_{2,2}&&\dots&&\GX_{r-2,r-2}&&\GX_{r-1,r-1}&&\GX_{r,r}\rlap{.}
\end{tikzcd}\]

The centers of the blowups $\prul,\prur$ appearing in Diagram (\ref{fig:Chow0}) have been described in \cite[Remark~5.10]{WORS6}; they are smooth for the maps appearing in the first layer right above the geometric quotients (see \cite[Lemma~5.7]{WORS6}):
\[\begin{tikzcd}[
  column sep={2.6em,between origins},
  row sep={2.9em,between origins}]
&\CX_{0,2}\arrow[ld,"\prul",labels=above left]\arrow[rd,"\prur"]&
&\CX_{1,3} \arrow[dl,"\prul",labels=above left]&&\dots&&\CX_{r-3,r-1}\arrow[rd,"\prur"] &
&\GX_{r-2,r}\arrow[ld,"\prul",labels=above left]\arrow[rd,"\prur"]&\\
\GX_{0,1}&&\GX_{1,2}&
&\dots&&\dots&
&\GX_{r-2,r-1}&&\GX_{r-1,r}\rlap{.}
\end{tikzcd}\]
For the maps in successive layers, the centers are not necessarily smooth, as shown in \cite[Example~6.1]{WORS6}). However, the situation improves under certain hypotheses. Let us recall the following. 

\begin{definition}\label{def:convex}
A smooth projective variety $X$ is said to be {\em convex} if for every morphism $\mu\colon\P^1\to X$ we have that $\HH^1(\P^1,\mu^*T_X)=0$. 
\end{definition}

The following  has been shown in \cite[Theorem~1.2]{WORS6}. 

\begin{theorem}\label{thm:WORS6smooth}
Let $(X,L)$ be as in Setup~\ref{set:WORS6}. Assume moreover that $X$ is convex. Then the varieties $\CX_{i,j}$ for $i=0$ or $j=r$ are smooth, and the maps 
\[\xymatrix{\GX_{0,1}&\CX_{0,2}\ar[l]_(0.44)s&\cdots\ar[l]_(0.42)s&\CX\ar[l]_(0.42)s\ar[r]^(0.44)d&\cdots\ar[r]^(0.40)d&\CX_{r-2,r}
\ar[r]^(0.46)d&\GX_{r-1,r}}\]
are blowups with smooth centers.
\end{theorem}

Let us illustrate Theorem~\ref{thm:WORS6smooth} in the examples quoted in the introduction. 

\begin{example}\label{ex:classicex1}
The space of complete collineations from a vector space of dimension $m$ to a vector space of dimension $n$, $n\leq m$, is the Chow quotient $\CX$ of the Grassmannian of $(n-1)$-dimensional linear subspaces of $\P^{n+m-1}$,  which in the language of marked Dynkin diagrams we denote by $X=\DA_{n+m-1}(n)$. We consider the action on this Grassmannian induced by the equalized action on $\P^{n+m-1}$ whose sink and source are  $\P^{n-1}$ and  $\P^{m-1}$. Following \cite{WORS5,FS}, one may easily obtain a description of all the fixed-point components of the action. In particular, one sees that the sink of the action on $X$ is a point, so $\GX_{0,1}$ is the projective space $\P(\C^n\otimes\C^m)$. Then $\CX_{0,2},\dots,\CX_{0,n}=\CX$ are obtained  by blowing up $\GX_{0,1}$ along the successive strict transforms of the traces of the closures of the BB-cells $X^+(Y_1), \dots, X^+(Y_{n-1})$. These are precisely the successive strict transforms of the classes of matrices of ranks $1,\dots,n-1$ in $\P(\C^n\otimes\C^m)$,  so we recover the classical description of the space of collineations appearing in \cite{Vai84}.
\end{example}

\begin{example}\label{ex:classicex2}
  If we start with a quadric of equation \[x_0x_{n}+\dots+x_{n-1}x_{2n-1}=0\] in $\P^{2n-1}$, endowed with the equalized $\C^*$-action whose sink and source are the linear subspaces
  \[Y_-:x_n=\dots=x_{2n-1}=0\quad\mbox{and}\quad Y_+:x_0=\dots=x_{n-1}=0,\]
respectively, we may consider the induced action on the spinor variety $X=\DD_n(n)$---more precisely,  the one parametrizing the deformations of $Y_-$. The point corresponding to $Y_-$ is the sink of the action on $\DD_n(n)$, and the inner fixed-point components, as described in \cite{FS}, are Grassmannians, $\DA_{n-1}(2s)$. The first geometric quotient $\GX_{0,1}$ is the projectivization of the tangent space of $X$ at $Y_-$, which is $\P(\bigwedge^2\C^n)$. As in the previous example, $\CX$ is then obtained by blowing up successively  this projective space along the  strict transforms of the traces of the closures of the BB-cells $X^+(Y_1), \dots, X^+(Y_{\lfloor(n-1)/2\rfloor})$, and one may check that these are precisely the strict transforms of the loci of classes of skew-forms $w\in \bigwedge^2\C^n$ of ranks $2,\dots, 2\lfloor(n-1)/2\rfloor$. This is the standard construction of the space of complete skew-forms, as presented in \cite{Thaddeus}.
\end{example}

\begin{example}\label{ex:classicex3}
  Similarly, we may start with a Lagrangian Grassmannian $X=\DC_n(n)$ parametrizing Lagrangian spaces $\P^{n-1}$ in a projective space $\P^{2n-1}$ with respect to a nondegenerate skew-symmetric form. We may assume that
  \[
  Y_-:x_n=\dots=x_{2n-1}=0\quad\mbox{and}\quad Y_+:x_0=\dots=x_{n-1}=0 
  \]
  are Lagrangian and consider the action on $X$ induced by the equalized $\C^*$-action on $\P^{2n-1}$ whose sink and source are $Y_\pm$. The sink of the action in $X$ is the point parametrizing $Y_-$, the geometric quotient $\GX_{0,1}$ is $\P(S^2\C^n)$, and the inner fixed components are Grassmannians $\DA_{n-1}(1),\dots,\DA_{n-1}(n-1)$ (see \cite{FS}). By a similar argument, we obtain that $\CX$ is the successive blowup of $\P(S^2\C^n)$ along the strict transforms of the classes of elements $q\in S^2\C^n$ of ranks $1,2,\dots, n-1$, which again coincides with the classical description given in \cite{Vai82}. 
\end{example}


\section{Chow quotients of convex varieties of Picard number 1}\label{sec:boundaryprimal}

Throughout the rest of the paper, we will work in the following setting. 

\begin{setup}\label{set:main}
Let $X$ be a smooth convex variety of Picard number $1$, endowed with an equalized nontrivial $\C^*$-action of criticality $r$. 
We will denote by  $\cM$ a family of rational curves of minimal degree on $X$ (which exists since $X$ is uniruled), by $p\colon\cU \to \cM$ its universal family, and by $q\colon \cU \to X$ the evaluation morphism (see \cite[Proposition~II.2.11]{kollar}).
\end{setup}

\begin{remark}\label{rem:convexcurves}
Since $X$ is convex, $\cM$ is {\em beautiful}; \textit{i.e.}, the evaluation morphism $q\colon \cU \to X$ is smooth and surjective (see \cite[Section 2.7]{OS2} for further details). We will denote by $\cU_x$ the fiber of $q$ over the point $x \in X$. Note that since $X$ is a Fano manifold (in particular, simply connected), then  every $\cU_x$ is connected, therefore irreducible (\textit{cf.}~\cite[Proposition~1.6]{Kan5}). 
\end{remark}

Let us recall some results recently obtained in \cite{OS2}. 
The $\C^*$-action on $X$ induces actions on $\cM$ and~$\cU$ such that $p$ and $q$ are equivariant morphisms. These actions are equalized (see \cite[Propositions 4.5 and 4.8]{OS2}); in particular, the closure of every  $1$-dimensional orbit in $\cM$ and $\cU$ is a smooth rational curve.

\begin{remark}\label{rem:f2}
Given a rational curve $\Gamma \subset \cM$ that is the closure of a $1$-dimensional orbit, $S_\Gamma:=p^{-1}(\Gamma)$ is a rational ruled surface $\FF_e$ with $e \le 2$ (see \cite[Proposition 4.6]{OS2}). By the same proposition, if no point of~$\Gamma$ parametrizes a curve that is fixed by the action on $X$, then $e=0$ or $e=2$ and the action has four fixed points (see \cite[Figure 4.1]{OS2}).
\end{remark}
 
 The following statement has been proved in \cite[Theorem 5.3]{OS2}.
 
 \begin{proposition}\label{prop:onecomp}
 In the situation of Setup~\ref{set:main}, denoting by $L$ the ample generator of\, $\Pic(X)$, the $L$-weights of the action are $0 < d < 2d < \dots < (r-1)d < rd = \delta$, and, for every $j =0, \dots, r$, there exists a unique $Y_j \in \cY$ such that $\mu_L(Y_j)=jd$.
\end{proposition}

\begin{remark}\label{rem:convex} 
 Proposition~\ref{prop:onecomp} allows us to use the following notation:
  \[
  B_j^\pm:=B^\pm(Y_j),\;\; \nu^\pm_j:=\nu^\pm(Y_j)\quad\mbox{for every $j=0, \dots, r$}\quad \mbox{and}\quad \nu^+_{-1}=\nu^-_{r+1}:=0.\] Note that, by definition, $\nu^+_0=\nu^-_r=0$, $\nu^-_0=\dim X-\dim Y_0$, and $\nu^+_r=\dim X-\dim Y_r$.
\end{remark}

\subsection{Boundary divisors}\label{ssec:boundary}

The main result of this section is that the subscheme of $\CX$ parametrizing reducible cycles is a union of smooth divisors, which we call {\em boundary divisors}, with simple normal crossings. 

Our first observation is the following. 

\begin{remark}\label{rem:Hilb2}
Since $X$ is a convex variety, $\CX$  is smooth by Theorem~\ref{thm:WORS6smooth}, which has been proven in \cite[Proposition~6.4]{WORS6}. The arguments there, applied directly to $\CX$, show when $\CX$ is considered as a subscheme of $\Hilb(X)$ (see Remark~\ref{rem:Hilb}), the convexity assumption implies that $\Hilb(X)$ is smooth at every point of $\CX$. Then $\CX$, which is the only $\C^*$-fixed-point component of $\Hilb(X)$ containing the class of the closure of the general orbit in $X$, is smooth as well.
\end{remark}

Given a closed point $[C]$ of $\CX\subset \Hilb(X)$, it determines a $\C^*$-invariant subscheme $C$ of $X$. Moreover, as observed in \cite[Proof of Proposition 6.4]{WORS6}, the singularities of $C$ are nodes. Let us now recall some facts about deformations of projective nodal curves. We refer the interested reader to \cite[Section~6]{Talpo} for a detailed account on this topic.

\begin{proposition}\label{prop:deforms}
Let $C$ be a projective connected curve whose irreducible components are smooth rational curves $C_1,\dots, C_k$ and whose singular locus is the set of points $\{P_i=C_{i}\cap C_{i+1}|\,\,i=1,\dots,k-1\}$. We assume that the singularities of\, $C$ at these points are ordinary nodes. Then 
\begin{itemize}[leftmargin=25pt]
\item $\cExt^1(\Omega_C,\cO_C)$ is supported on the singular points $P_i$, $i=1,\dots,k-1$, and $$\cExt^1(\Omega_C,\cO_C)_{P_i}\simeq \C$$
for every $i=1,\dots,k-1$;
\item the complex vector space $\Ext^1(\Omega_C,\cO_C)$ of first-order deformations of $C$  has dimension $k-1$, and it is naturally isomorphic to
\[
\HH^0(C,\cExt^1(\Omega_C,\cO_C))\simeq \bigoplus_{i=1}^{k-1}\cExt^1(\Omega_C,\cO_C)_{P_i}.
\]
\end{itemize}
\end{proposition}

\begin{proof}
For the first assertion, we refer the reader to \cite[Remark~6.4.2]{Talpo}. The second claim follows from the local-to-global
$\Ext$ spectral sequence.
\end{proof}

\begin{remark}\label{rem:deforms}
  Proposition~\ref{prop:deforms} can be interpreted geometrically as follows. The first-order deformations of $C$ correspond to smoothings of $C$ at its singular points (see \cite[Sections~6.3 and~6.4]{Talpo} for details).  More precisely, for every $i$, the $1$-dimensional vector space $\cExt^1(\Omega_C,\cO_C)_{P_i}$ corresponds, in the space $\Ext^1(\Omega_C,\cO_C)$ of first-order deformations of $C$, to deformations that smoothen the point $P_i$ and are locally trivial around the $P_j$ with $j\neq i$.

In this way, when we choose isomorphisms $\cExt^1(\Omega_C,\cO_C)_{P_i}\simeq \C$ and denote by $t_i$, $i=1,\dots , k-1$, the compositions
$t_i\colon\Ext^1(\Omega_C,\cO_C)\to \cExt^1(\Omega_C,\cO_C)_{P_i}\simeq \C$,  the deformations of $C$ that are locally trivial around $P_i$ are given by the equation $t_i=0$.
\end{remark}

We will now go back to the case in which $C$ is the subscheme of $X$ determined by a closed point $[C]\in \CX$ and see that every first-order deformation of  $C$ extends to a deformation of $C$ as a $\C^*$-invariant subscheme of $X$. More precisely, we will show the following. 

\begin{proposition}\label{prop:kodspen}
Let $[C]$ be a closed point of\, $\CX$. Then  the restriction to $T_{\CX,[C]}$ of the Kodaira--Spencer map 
\[
T_{\Hilb(X),[C]}\lra \Ext^1(\Omega_C,\cO_C)
\]
is surjective.
\end{proposition}

\begin{proof}
As observed in \cite[Proof of Proposition 6.4]{WORS6}, the singularities of $C$ are local complete intersections; then we have a short exact sequence
\begin{equation}\label{eq:ses1}
\shse{\cI/\cI^2}{{\Omega_X}_{|C}}{\Omega_C}
\end{equation}
and  $\Ext^1({\Omega_X}_{|C},\cO_C)=0$. 
We get the exact sequence 
\begin{multline}\label{eq:FP}
0\lra\Hom(\Omega_C,\cO_C)\lra\Hom({\Omega_X}_{|C},\cO_C)\lra \Hom(\cI/\cI^2,\cO_C)\lra \Ext^1(\Omega_C,\cO_C)\lra 0.
\end{multline}
The space $\Hom(\cI/\cI^2,\cO_C)$ is the Zariski tangent space of $\Hilb(X)$ at $[C]$, and the map $\Hom(\cI/\cI^2,\cO_C)\to \Ext^1(\Omega_C,\cO_C)$ is the Kodaira--Spencer map, associating  with every first-order deformation of $C$ as a subscheme of $X$ the corresponding first-order deformation of $C$ as a scheme over $\Spec(\C)$ (\textit{cf.}~\cite[Remark 2.4.4]{Ser1}). We then see that this map is surjective. 

Now we add the action of $\C^*$ into the picture. On one hand, we have induced $\C^*$-actions on the  spaces $\Ext^1(\Omega_C,\cO_C)$ and $\HH^0(C,\cExt^1(\Omega_C,\cO_C))$, and the natural map between them is $\C^*$-equivariant. Hence, since $\cExt^1(\Omega_C,\cO_C)$ is supported on the singular points of $C$, we conclude that the action on $\Ext^1(\Omega_C,\cO_C)$ is trivial.

 On the other hand, the tangent space $T_{\CX,[C]}$ is the weight zero eigenspace  $\Hom(\cI/\cI^2,\cO_C)_0\subset \Hom(\cI/\cI^2,\cO_C)$, which must surject onto the weight zero eigenspace $\Ext^1(\Omega_C,\cO_C)_0=\Ext^1(\Omega_C,\cO_C)$. This completes the proof.
\end{proof}

\begin{corollary}\label{cor:kodspen}
In the situation of Setup~\ref{set:main}, the subscheme $\CX_{\sing}$ of\, $\CX$ parametrizing singular $\C^*$-invariant subschemes is a divisor with normal crossings.
\end{corollary}

\begin{proof}
The fact that $\CX_{\sing}$ is a divisor follows from \cite[Propositions 6.2.11 and 6.3.5]{Talpo}.  We will now study its local equations.

Let $[C]$ be a closed point, corresponding to a singular $\C^*$-invariant subscheme $C\subset X$, with irreducible components $C_1,\dots, C_k$ and singular points $P_i=C_i\cap C_{i+1}$, $i=1,\dots,k-1$. Let us  consider again the Zariski tangent space $T_{\CX,[C]}$ and the linear map
$$
\phi\colon T_{\CX,[C]}\lra \Ext^1(\Omega_C,\cO_C)\simeq \bigoplus_i\cExt^1(\Omega_C,\cO_C)_{P_i}
$$ 
defined as the restriction to $T_{\CX,[C]}$ of the Kodaira--Spencer map, which is surjective by Proposition~\ref{prop:kodspen}. We pull back via $\phi$ the linear coordinates $(t_1,\dots, t_{k-1})$ in  $\Ext^1(\Omega_C,\cO_C)$ introduced in Remark~\ref{rem:deforms}; abusing notation, we denote these linear forms by $t_i\colon T_{\CX,[C]}\to\C$.  This tells us that locally around $[C]$ the divisor $\CX_{\sing}\subset\CX$ is given by an equation of the form $\ol{t}_1\cdot\ldots \cdot \ol{t}_{k-1}=0$, where $(d\ol{t}_i)_{[C]}=t_i$.  In particular, it has normal crossings at $[C]$. 
\end{proof}

To finish our description of $\CX_{\sing}$,  we introduce the following divisors.  

\begin{definition}\label{def:Ei}
For  every $i\in \{1,\dots,r-1\}$, we denote by $E^\circ_i\subset\CX$ the set parametrizing cycles having a unique singular point in the component $Y_i$. By Corollary~\ref{cor:CXcontainsall} and the Bia{\l}ynicki-Birula theorem, $E^\circ_i$ is isomorphic to an open subset of $\P(N_{Y_i}^-)\times_{Y_i}\P(N_{Y_i}^+)$, which has dimension $n-1$. We denote by $E_i$ its closure in $\CX$ and call it the $\supth{i}$ {\em boundary divisor of }\, $\CX$. 
\end{definition}

\begin{theorem}\label{thm:smoothborder} 
Let $X$ be a convex variety of Picard number $1$, endowed with an equalized $\C^*$-action of criticality $r\geq 2$. Then the Chow quotient $\CX$ of\, $X$ by the action of\, $\C^*$ is smooth, and the subscheme $\CX_{\sing}\subset\CX$ parametrizing reducible elements is a divisor with simple normal crossings, equal to
$$\sum_{i=1}^{r-1}E_i.$$
Moreover, we have $[C] \in E_i$ if and only if the cycle $C$ has a singular point in $Y_i$. 
\end{theorem}

\begin{remark}\label{rem:main1}
Note that in the case in which the criticality of the action is $r=1$, Theorem~\ref{thm:smoothborder} is obvious. In fact, in this case the Chow quotient $\CX$ equals the geometric quotient $\GX_{0,1}$, which is the projectivization of the normal bundle of $Y_0$ in $X$, hence smooth, and the subscheme of $\CX$ parametrizing reducible $1$-cycles is empty.
\end{remark}

\begin{proof}[Proof of Theorem~\ref{thm:smoothborder}]
Take an element $[C]\in \CX$ consisting of $k$ irreducible components $C_1,\dots,C_k$, and denote by $P_1\in Y_{i_1},\dots, P_{k-1}\in Y_{i_{k-1}}$ its singular points. 

With the notation of the proof of Corollary~\ref{cor:kodspen}, for every $j\in\{1,\dots,k-1\}$, there exists a smooth  curve $\gamma_j$ in $\CX$ passing through
 $[C]$, whose tangent space at $[C]$ is generated by $v\in T_{\CX,[C]}$, such that $\phi(v)\in \Ext^1(\Omega_C,\cO_C)$ satisfies $t_j(\phi(v))=1$ and $t_s(\phi(v))=0$ for $s\neq j$. 

The general element of $\gamma_j$ has a unique singular point, which belongs to $Y_{i_j}$; in fact, the singular points of an element of $\CX$ are fixed points, and the unique singular point of an element of $\gamma_j$ different from $C$ is a deformation of $P_j\in Y_{i_j}$. In particular, the general element of $\gamma_j$ belongs to  $E_{i_j}$, so $[C]\in E_{i_j}$ as well. We conclude that $[C]\in \bigcap_{j=1}^{k-1}E_{i_j}$.

On the other hand, the arguments in the proof of Corollary~\ref{cor:kodspen} tell us that, locally around $[C]$, $\CX_{\sing}$ is the normal crossing intersection of exactly $k-1$ divisors; hence these divisors are necessarily the $E_{i_j}$.  This finishes the proof.
\end{proof}

\begin{remark}\label{rem:smoothborder}
Theorem~\ref{thm:smoothborder} is not true without the convexity assumption: In general, the subscheme of $\CX$ parametrizing reducible cycles may contain components different from the $E_i$ defined above (see \cite[Example~6.1]{WORS6}). 
\end{remark}

We now relate the boundary divisors with the exceptional divisors of the blowups appearing in Diagram~(\ref{fig:Chow0}). Note that the centers of the blowups $\prur\colon\CX_{i-1,i+1}\to \GX_{i,i+1}$ and $\prul\colon\CX_{i-1,i+1}\to \GX_{i-1,i}$ are smooth by  \cite[Lemma~5.7]{WORS6}. 

\begin{lemma}\label{lem:Eiblowup}
For every $i = 1, \dots, r-1$, let $E'_i$ be the exceptional divisor of the blowups
$\prur\colon\CX_{i-1,i+1}\to \GX_{i,i+1}$, $\prul\colon\CX_{i-1,i+1}\to \GX_{i-1,i}$. Then the pullback of $E'_i$ to $\CX$ is the boundary divisor $E_i$.
\end{lemma}

\begin{proof}
Let  $[C]\in \CX$ be a cycle; its image in $\CX_{i-1,i+1}$ can be described as follows:  We take the intersection of $C$ with the open set $X\setminus (B^+_{i-1}\cup B^-_{i+1})$, which maps isomorphically into the pruning $X_{i-1,i+1}$, and we consider its closure in $X_{i-1,i+1}$; this  determines an element $[C'] \in \CX_{i-1,i+1}$.  The divisor $E'_i$ is isomorphic to $\P(N_{Y_i}^-)\times_{Y_i}\P(N_{Y_i}^+)$; hence $[C'] \in E_i'$ if and only if $C$ is singular at a point of $Y_i$,  By Theorem~\ref{thm:smoothborder}, this happens if and only if $[C]\in E_i$. 

This shows that the inverse image of $E'_i$ (which is the support of the pullback of $E'_i$ to $\CX$) is $E_i$. It also follows that the birational map $\CX\to \CX_{i-1,i+1}$ is a local isomorphism at a general point of $E_i$; hence  $E_i$ is the pullback of $E'_i$.
\end{proof}

As a consequence of  Lemma~\ref{lem:Eiblowup}, considering the commutative diagram with Cartesian squares
\[
\begin{tikzcd}[
  column sep={4em,between origins},
  row sep={3em,between origins}]
&&\CX_{0,r}\arrow[rd,"\prur"] \arrow[dl,"\prul",labels=above left] &&\\
&\CX_{0,i+1}\arrow[rd,"\prur"] \arrow[dl,"\prul",labels=above left]&&\CX_{i-1,r}\arrow[rd,"\prur"] \arrow[dl,"\prul",labels=above left]&\\
\CX_{0,i}\arrow[rd,"\prur"]&&\CX_{i-1,i+1}\arrow[rd,"\prur"] \arrow[dl,"\prul",labels=above left]&&\CX_{i,r} \arrow[dl,"\prul",labels=above left]\\
&\GX_{i-1,i}&&\GX_{i,i+1}\rlap{,}&
\end{tikzcd}
\]
we see that $E_i$ equals both the pullbacks to $\CX=\CX_{0,r}$ of the exceptional divisors of the smooth blowups $\prul\colon \CX_{0,i+1}\to \CX_{0,i}$ and $\prur\colon \CX_{i-1,r}\to \CX_{i,r}$. 

\subsection{Smoothness of the partial Chow quotients}\label{ssec:smoothpart}

Here we will show that, in the situation of Setup~\ref{set:main}, all the maps $\prul,\prur$ appearing in Diagram (\ref{fig:Chow0}) are blowups with smooth centers.

\begin{theorem}\label{thm:smoothpart}
Let $X$ be as in Setup~\ref{set:main}. Then, for $i<j$, the Chow quotients $\CX_{i,j}$  are smooth, and  the maps $\prul\colon \CX_{i,j+1}\to \CX_{i,j}$ and  $\prur\colon \CX_{i-1,j}\to \CX_{i,j}$  are blowups with smooth irreducible centers. 
\end{theorem}

The proof of Theorem~\ref{thm:smoothpart} will be obtained by using recursively the following. 

\begin{lemma}\label{lem:smoothblow}
Consider a Cartesian diagram of normal projective varieties 
\[
\xymatrix@=12pt{&W\ar[ld]_{\prul}\ar[rd]^{\prur}\ar[dd]^\varphi&\\Y_1\ar[rd]_{\prur'}&&Y_2\ar[ld]^{\prul'}\\&Z&}
\]
such that
\begin{itemize}
\item $W,Y_1,Y_2$ are smooth, and the morphisms $\prul$, $\prur$ are blowups with smooth irreducible centers and exceptional loci $E_\prul, E_\prur$;
\item $\varphi(E_\prul) \cap \varphi(E_\prur)$ is a proper subset of\, $\varphi(E_\prul)$ and of\, $\varphi(E_\prur)$.
\end{itemize}
Then $Z$ is smooth, and the maps $\prur'$, $\prul'$ are blowups with smooth  centers.
\end{lemma}

\begin{proof}
It is clearly enough to show that $\prur'$ is a blowup with smooth center. By our second assumption, $\prul(E_\prur)$ is not contained in $\prul(E_\prul)$; hence we may assert that $\prur'$ is a divisorial contraction with exceptional locus $\prul(E_\prur)$. This is elementary because the relative Picard number $\rho(W/Z)$ is $2$, since $\varphi$ contracts a $2$-dimensional face of the Mori cone of $W$.

Moreover, there exists a point $P\in \prur'(\Exc(\prur'))\setminus \prur'(\prul(E_\prul))$ such that $\prul$ is a local isomorphism at the points of  the fiber $F=\varphi^{-1}(P)$. Since the diagram is Cartesian and $\prur$ is a smooth blowup, then $F$ is a projective space. Therefore, 
\[
\left(-K_{Y_1}\right)_{|\prul(F)}=\left(-\prul^*K_{Y_1}\right)_{|F}=\left(-K_{W}\right)_{|F}=\dim F=\dim \prul(F).
\]
Then, by \cite[Theorem~5.1]{AO2}, it follows that $\prur'$ is a blowup with a smooth irreducible center.
\end{proof}

\begin{proof}[Proof of Theorem~\ref{thm:smoothpart}]
From Theorem~\ref{thm:WORS6smooth}, we already know that the assertion is true for $i=0$ and $j=r$, that is, for the varieties and maps appearing on the upper sides of Diagram (\ref{fig:Chow0}). Now we apply Lemma~\ref{lem:smoothblow} recursively on $r-i+j$. 

At a given step, if  $\prul\colon \CX_{i-1,j+1}\to \CX_{i-1,j}$ and $\prur\colon \CX_{i-1,j+1}\to \CX_{i,j+1}$ are blowups with smooth centers, we consider the diagram 
\[
\xymatrix@R=7pt@C=5pt{&\CX_{i-1,j+1}\ar[ld]_{\prul}\ar[rd]^{\prur}&\\\CX_{i-1,j}\ar[rd]&&\CX_{i,j+1}\ar[ld]\\&\CX_{i,j}\rlap{.}&}
\]
The exceptional divisors of $\prul$, $\prur$ are the images in $\CX_{i-1,j+1}$ of $E_j,E_i\subset \CX$, respectively. In particular, they are irreducible. By Theorem~\ref{thm:smoothborder}, the intersection of these two sets is properly contained in both. This concludes the proof.
\end{proof}

\subsection{Primal curves}\label{ssec:primal}

In this section, we construct some special curves in $\CX$; we will show later that their classes generate the Mori cone of $\CX$.  The construction is based on the following.

\begin{lemma}\label{lem:smoothij}
For every $i=0, \dots,r-1$ and for every point $P_{i-1} \in Y_{i-1}$, there exist a point $P_{i+1} \in Y_{i+1}$ and a curve $\Gamma \subset \cM$, which is the closure on an orbit, such that $S=p^{-1}(\Gamma)$ is a ruled surface $\FF_2$, the sink and the source of the action on $S$ are mapped by the evaluation map $q\colon \cU\to X$ to $P_{i-1}$ and to $P_{i+1}$, respectively, and the minimal section is contracted by $q$ to a point $P_i \in Y_i$.
\end{lemma}

\begin{proof}
Take a general $\C^*$-invariant curve $C_-$ in $\cM$ connecting $P_{i-1}\in Y_{i-1}$ to $Y_i$, and denote its source by $P_{i}\in Y_i$.  The existence of $C_-$ is guaranteed by \cite[Theorem~5.3]{OS2}. Now consider the action of $\C^*$ on the invariant subvariety $p(\cU_{P_i})\subset\cM$. By construction, the element $[C_-]$ belongs to the sink of the action on $p(\cU_{P_i})$. Let $\Gamma \subset p(\cU_{P_i})$ be the closure of a general orbit with sink $[C_-]$. Its source $[C_+]$ parametrizes a $\C^*$-invariant curve $C_-$ in $\cM$ with sink $P_i$ and source $P_{i+1} \in Y_{i+1}$.  In particular, we have constructed a $\C^*$-invariant connected cycle $C_{\mbox{\scriptsize{s}}} = C_-+C_+$ linking $P_{i-1}\in Y_{i-1}$ to $P_{i+1}\in Y_{i+1}$.

Consider $S= p^{-1}(\Gamma)$; since the induced action on $\cM$ is equalized (see \cite[Proposition 4.5]{OS2}), $S$ is a smooth ruled surface. Since $C_\pm$ are not fixed and $P_i \in C$, by \cite[Proposition 4.6]{OS2}, we have $S \simeq \FF_2$. The action is the one described in Remark~\ref{rem:f2}.
\end{proof}

Let us set
\begin{equation}\label{eq:I}
I :=\{k\,|\, \dim Y_k >0\}\subseteq \{0,\dots, r\}.
\end{equation}

\begin{construction}\label{con:gij}
Given $i\neq 0,r$, we consider a point $P_{i-1}\in Y_{i-1}$ and the $1$-dimensional family of orbits linking $P_{i-1}$ to a point $P_{i+1}$ in $Y_{i+1}$ constructed in Lemma~\ref{lem:smoothij}. Closing these orbits and adding  
$\C^*$-invariant cycles $\Delta_0$ linking $P_{i-1}$ to the sink and $\Delta_r$ linking $P_{i+1}$ to the source, we obtain  a curve $\Gamma_i\subset \CX$. 

In fact, the Chow quotient of the ruled surface $\FF_2$ with respect to the $\C^*$-action described in Remark~\ref{rem:f2} is isomorphic to the projectivization of the tangent space at its sink, \textit{i.e.}, to $\P^1$, and the equivariant morphism
$\FF_2\to X$ given by Lemma~\ref{lem:smoothij} induces a morphism $\PP^1 \to \CX_{i-1,i+1}$. 
Moreover, we have two constant maps $\PP^1\to \CX_{0,i-1}$ and $\PP^1\to \CX_{i+1,r}$, defined by the choice of $\Delta_0$ and $\Delta_r$. Since $\CX$ is the normalized inverse limit of Diagram (\ref{fig:Chow0}), we get a morphism from $\PP^1$ to $\CX$. 

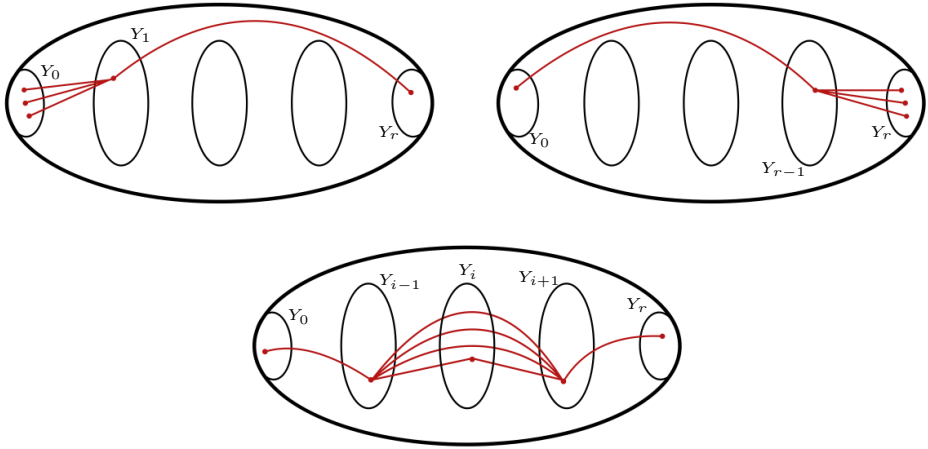
\begin{figure}[h!]
\begin{tikzpicture}
\node at (0,0) {\includegraphics[width=12.5cm]{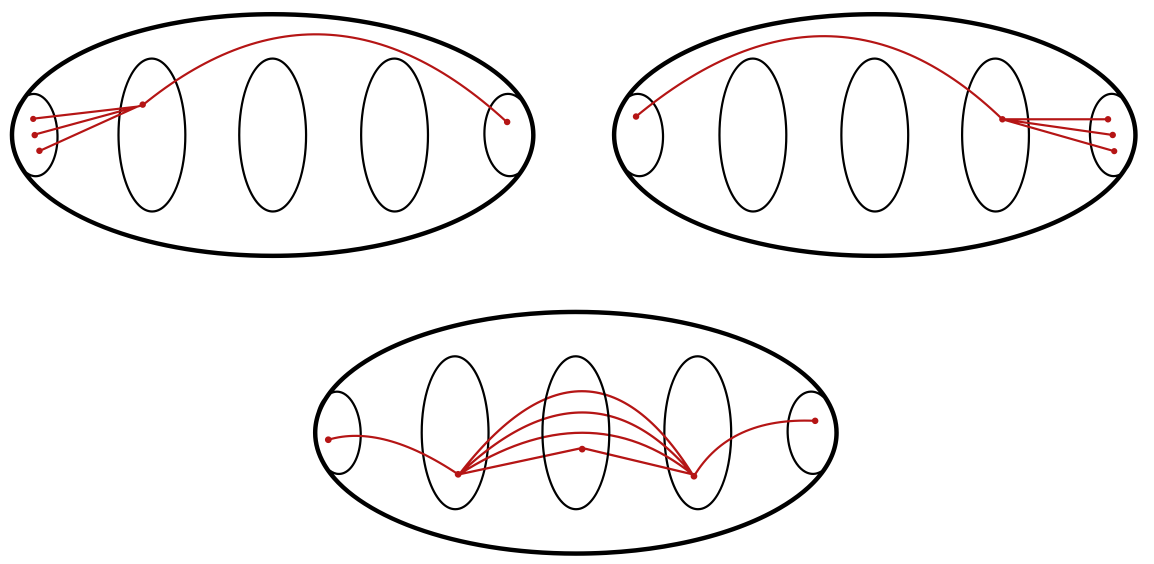}};
\node at (-5.55,2) {\tiny $Y_0$};
\node at (-4.35,2.5) {\tiny $Y_1$};
\node at (-1.05,1.2) {\tiny $Y_r$};
\node at (0.95,1.1) {\tiny $Y_0$};
\node at (4.2,0.7) {\tiny $Y_{r-1}$};
\node at (5.5,1.2) {\tiny $Y_r$};
\node at (-2.25,-1.25) {\tiny $Y_0$};
\node at (-0.9,-0.8) {\tiny $Y_{i-1}$};
\node at (0,-0.65) {\tiny $Y_i$};
\node at (0.95,-0.75) {\tiny $Y_{i+1}$};
\node at (2.25,-1.1) {\tiny $Y_r$};
\end{tikzpicture}
\caption{The $1$-dimensional families of $\C^*$-invariant cycles determining the curves $\Gamma_0,\Gamma_r,\Gamma_i$ ($i=1,\dots,r-1$).\label{fig:Curve2}}
\end{figure}

In a similar way, if $0$ or $r$ belongs to $I$, we can construct rational curves $\Gamma_0, \Gamma_r \subset \CX$.  Let us explain the case $i=0$ (the case $i=r$ is analogous).

If $0 \in I$, we consider a point $P_1\in Y_1$ and a line $\ell\subset\P(\cN(Y_1)^+_{P_1})$. For every element of $\ell$, we have a $\C^*$-invariant curve in $\cM$ linking $P_1$ to $Y_0$. Adding a fixed $\C^*$-invariant cycle $\Delta_r$ linking $P_1$ and the source, we obtain a $1$-dimensional family of $\C^*$-invariant cycles, parametrized by a rational curve $\Gamma_0\subset \CX$.
\end{construction}

\begin{definition}\label{def:curves}
We will call {\em primal curves} in $\CX$ the curves $\Gamma_i, i \in I$, described in Construction~\ref{con:gij}. We have represented the families of cycles parametrized by these curves in Figure~\ref{fig:Curve2}.
\end{definition}

\subsection{Intersection numbers}

In this section, we compute the intersection matrix of primal curves and boundary divisors. 

We first note that the cycles parametrized by $\Gamma_0$ (respectively, $\Gamma_r$) are singular only at points of $Y_1$ (respectively, $Y_{r-1}$); hence $ E_i\cdot\Gamma_{0}=0$ for $i\geq 1$ (respectively, $E_i\cdot\Gamma_{r}=0$ for $i\leq r$). On the other hand, the cycles parametrized by $\Gamma_j$ can only be singular at points of $Y_{j-1},Y_j,Y_{j+1}$; therefore, $E_i\cdot \Gamma_j=0$ for $i\neq j-1,j,j+1$.  We are left with computing the values $a_{ij}$ for $i=j$, $i=j-1$, and $i=j+1$.

\begin{lemma}\label{rem:intnumbers2}
We have $E_i\cdot\Gamma_j=-1$ if\, $j=i-1,i+1$.
\end{lemma}

\begin{proof}
We will show the statement in the case $j=i+1$; the case $j=i-1$ is analogous. 
We consider the following diagram, with  Cartesian squares:
$$
\xymatrix@=3.7mm{&&\CX_{0,r}\ar[ld]_{\prul}\ar[rd]^{\prur}&\\
&\CX_{0,i+1}\ar[rd]^{\prur}&&\CX_{i-1,r}\ar[rd]^{\prur}\ar[ld]_{\prul}\\
&&\CX_{i-1,i+1}\ar[rd]^{\prur}&&\CX_{i,r}\ar[ld]_{\prul}\\&&&\GX_{i,i+1}\rlap{.}&}
$$
We claim that the image  $\Gamma'_{i+1}$ of  $\Gamma_{i+1}$ in $\CX_{i-1,i+1}$, via the isomorphism $E'_i\simeq \P(N^+(Y_i))\times_{Y_i}\P(N^-(Y_i))$ given by the Bia{\l}ynicki-Birula theorem, is  a line in a fiber of the projection $\P(N^+(Y_i))\times_{Y_i} \P(N^-(Y_i))\to \P(N^+(Y_i))$. 

In fact, the cycles of the family $\Gamma_{i+1}$ are constructed by means of an equivariant morphism $\FF_2\to X$, so  the sink $x_-$ of the action on $\FF_2$ maps to a point $P_i\in Y_i$ contained in all the cycles parametrized by $\Gamma_{i+1}$. Since the action of $\C^*$ on $X$ is equalized, the cycles parametrized by $\Gamma_{i+1}$ must have different tangent directions at~$P_i$, and we may then conclude that the differential of the map $\FF_2\to X$ at $x_-$ is injective. The line obtained as the projectivization of its image, which lies in $\P(N^-(Y_i)_{P_i})$, is the curve $ \Gamma'_{i+1}$.

By Lemma~\ref{lem:Eiblowup}, $E_i$ is the pullback to $\CX$ of the exceptional divisor $E'_i$ of the map $\prur\colon \CX_{i-1,i+1}\to \GX_{i,i+1}$. On the other hand, $\Gamma'_{i+1}$ is a line in a fiber of the contraction of $E'_i\simeq \P(N^-(Y_{i}))\times_{Y_{i}}\P(N^+(Y_{i}))\to \P(N^+(Y_{i}))$. Since the map $\prur$ is a blowup with a smooth center, we conclude that
\[
E_i\cdot \Gamma_{i+1}=E'_i\cdot \Gamma'_{i+1}=-1. \qedhere
\]
\end{proof}

\begin{lemma}\label{lem:intnumbers1}
We have $E_i\cdot\Gamma_i=2$ for every $i=1,\dots,r-1$.
\end{lemma}

\begin{proof}
Consider a curve $\Gamma_i$ in $\CX$. There exists precisely one element $C_{\mbox{\scriptsize{s}}}$ in this family of cycles meeting the fixed-point component $Y_i$ (see Lemma~\ref{lem:smoothij}). Let us denote the intersection point by $P_i$ and note that two components of $C_{\mbox{\scriptsize{s}}}$ are $\C^*$-invariant curves of $\cM$ linking $P_i$ to points $P_{i-1}\in Y_{i-1}$ and  $P_{i+1}\in Y_{i+1}$, respectively. We now use \cite[Theorem~2.5]{BB} to claim that $X$ is locally $\C^*$-equivariantly isomorphic to $U\times {N_{Y_i|X,P_i}}$ for an open neighborhood $U$ of $P_i$ in $Y_i$. 

We consider linear coordinates $(x_1,\dots,x_{\nu_i^-})$ and $(y_1,\dots,y_{\nu_i^+})$  in the normal spaces $N(Y_i)_{P_i}^-$, $N(Y_i)^+_{P_i}$, respectively, and linear coordinates $z_1,\dots,z_k$  of $Y_i$ around $P_i$. We can appropriately choose these sets of coordinates so that, on a neighborhood of $P_i$, the cycles parametrized by the curve $\Gamma_i$ have equations
\[
x_1y_1+z_1^2=z_1+c=0\quad\text{and}\quad x_i=y_i=z_i=0 \quad (i\geq 2),
\] 
where $c$ is a complex parameter. This family (whose parameter space is an open set of $\Gamma_i$) contains a unique reducible conic (given by $c=0$). We will compute the multiplicity of intersection of $\Gamma_i$ with the set of reducible conics in the $3$-dimensional space of coordinates $x_1,y_1,z_1$ at $c=0$.

 In order to do so, we first note that the linear span of a deformation of a conic of equations $x_1y_1+z_1^2=z_1+c=0$ has equation of the form $z_1+c=ax_1+by_1$ with $a,b\in \C$. In each of these planes, we have a distinguished conic with equations $x_1y_1+z_1^2=z_1+c-ax_1-by_1=0$. By adding the deformations of these conics within their linear span, we get a parametrization of the deformations of our conic $(x_1y_1=z_1^2=0)$ in the $3$-dimensional space:
\[
x_1y_1+z_1^2+\left(dx_1^2+ey_1^2+gx_1+hy_1+f\right)=z_1+c-ax_1-by_1=0,\quad a,b,c,d,e,f,g,h\in \C,
\]
which can be written as 
\[
\begin{pmatrix}x_1&y_1&1\end{pmatrix}
\begin{pmatrix}
d+a^2 & 1+ab & g-ac\\
1+ab & e+b^2 & h-bc\\
g-ac & h-bc & f+c^2
\end{pmatrix}
\begin{pmatrix}x_1\\y_1\\1\end{pmatrix}=z_1+c-ax_1-by_1=0.\]
Such a conic is reducible if and only if
\[
-f-c^2+(\mbox{terms of degree $\geq 3$})=0. 
\]
Restricting this equation to  $\Gamma_i$ (which is given by $a=b=d=e=f=g=h=0$), we get $c^2=0$.
\end{proof}

Summing up, we may write the intersection numbers of the divisor $E_i$ and the curves $\Gamma_j$ in the matrix in the following statement (note that $\Gamma_0$ and $\Gamma_r$ may not be defined if the sink and the source are isolated points). 

\begin{proposition}\label{prop:intnumbers}
Let $X$ be as in Setup~\ref{set:main}.  Then the  intersection matrix of the boundary divisors and the primal curves 
is the following:
\[
[E_i \cdot \Gamma_j]=
\begin{pmatrix}-1&2&-1&0&\dots&0&0&0\\
0&-1&2&-1&\dots&0&0&0\\
0&0&-1&\ddots&\ddots&\vdots&\vdots&\vdots\\
\vdots&\vdots&\vdots&\ddots&2&-1&0&0\\
0&0&0&\dots&-1&2&-1&0\\
0&0&0&\dots&0&-1&2&-1\end{pmatrix}
\]
with $i=1,\dots, r-1$ and $j = 0,\dots,r$.
\end{proposition}


\section{Nef and Mori cones of the Chow quotient}\label{sec:nefmori}

We want now to describe the nef cone and the cone of curves of $\CX$, for a variety $X$ as in Setup~\ref{set:main}.
Let us start by introducing some line bundles on $\CX$:
\begin{notation}\label{not:divisors}
For every $k\in \{0,\dots,r\}$, let $\cL_{k,k} \in \Pic(\CX)$ be the pullback of an ample line bundle $\cH_{k,k}$ on $\GX_{k,k}$. 
Note that $\cL_{0,0}$ and $\cL_{r,r}$ may be trivial (if $Y_0$ or $Y_r$, respectively, are isolated points), and that the set $I=\{k\,|\, \dim Y_k >0\}\subseteq \{0,\dots, r\}$ introduced in (\ref{eq:I}) can be written as $I =\{k\,|\, \cL_{k,k}\not\simeq \cO_{\CX}\}$.

We set
\[
\cA:=\bigotimes_{k} \cL_{k,k}\quad\mbox{and}\quad \cD_i:=\bigotimes_{k \not =i} \cL_{k,k}=\cA\otimes\cL_{i,i}^{-1} \quad\mbox{for every $i\in \{0,\dots,r\}.$}\]
As in Definition~\ref{def:Ei}, we will denote by $E_i\subset\CX$ the closure of the set parametrizing cycles having a unique singular point in $Y_i$, $i\in \{1,\dots,r-1\}$; note that, by  Theorem~\ref{thm:smoothborder},  $E_i$ parametrizes {\em all} the cycles having a singular point in $Y_i$, and  $\sum E_i$ is a simple normal crossing divisor, whose support is the subset $\CX_{\sing}\subset \CX$ parametrizing singular $\C^*$-invariant cycles. In particular, every $E_i$ is smooth. Abusing notation, we will denote by $E_i\in \Pic(\CX)$  the associated line bundles as well. Moreover, we will denote by $\cL_{k,k},\cA,\cD_{i},E_i\in \NU(\CX)$ the corresponding numerical classes of these divisors. 
\end{notation}

We start by determining the Picard number of $\CX$ and a basis of $\NU(\CX)$. 

\begin{lemma}\label{lem:pic} 
The Picard number of\, $\CX$ is $|I|$, and $\NU(\CX)$
 is generated by the classes  $\cL_{i,i}$, $i \in I$. 
 \end{lemma}

\begin{proof}
Let us first assume that $\dim Y_0, \dim Y_r >0$.  By \cite[Theorem~1.2 and Remark 5.10]{WORS6}, $\CX$ is obtained from the geometric quotient $\GX_{0,1}$ via a sequence of blowups
\[
\xymatrix{\GX_{0,1}&\ar[l]_{\prul}\CX_{0,2}&\ar[l]_{\prul}\quad\cdots\quad&\ar[l]_{\prul}\CX_{0,r}=\CX}
\]
 with smooth centers $S_{0,j}\subset \CX_{0,j}$, $j=1,\dots,r-1$. Since there is only one fixed-point component of the action on $X$ for each weight, the centers $S_{0,j}$ are irreducible; hence, for each $j=1, \dots,r-1$, the contraction $s\colon \CX_{0,j+1} \to \CX_{0,j}$ is elementary. Then the statement about the Picard number follows by observing that $\rho(\CX_{0,1})=2$; in fact, $\CX_{0,1}=\GX_{0,1}=\P_{Y_0}(\cN_{Y_0,X})$ and $\rho(Y_0)=1$ by \cite[Lemma 2.8]{WORS1}.

By abuse of notation, we will denote by $\cL_{i,i}$ the class of the pullback of $\cH_{i,i}$ to any Chow quotient $\CX_{0,j}$, $j\geq i$. To show the linear independence of $\{\cL_{i,i}\}_{i \in I}\subset\NU(\CX)$, we proceed by induction. First we observe that $\cL_{0,0}$ and $\cL_{1,1}$ are independent in $\NU(\CX_{0,1})$ since they support two different contractions of $\CX_{0,1}$.
Now assume  that $\{\cL_{i,i}\}_{i \le k-1}$ is linearly independent in $\NU(\CX_{0,k-1})$;  since the map $\NU(\CX_{0,k-1}) \to \NU(\CX_{0,k})$ is injective, $\{\cL_{i,i}\}_{i \le k-1}$ is linearly independent in $\NU(\CX_{0,k})$. To prove that $\{\cL_{i,i}\}_{i \le k}$ is linearly independent in $\NU(\CX_{0,k})$, it is then enough to notice that, for every $i \le k-1$, $\cL_{i,i}$ is trivial on the fibers of $\CX_{0,k} \to \CX_{0,k-1}$, while $\cL_{k,k}$ is ample on them. In fact, $\CX_{0,k}$ is the normalization of $\CX_{0,k-1}\times_{\CX_{k-1,k-1}}\CX_{k-1,k}$, so the fibers of $\CX_{0,k} \to \CX_{0,k-1}$ are, up to a finite morphism, inverse images of fibers of $\CX_{k-1,k}\to\CX_{k-1,k-1}$, which is an elementary contraction different from $\CX_{k-1,k}\to\CX_{k,k}$.

If $\dim Y_0 = 0$ and $\dim Y_r >0$ (resp.\ $\dim Y_0 > 0$ and $\dim Y_r =0$, resp.\ $\dim Y_0 = 0$ and $\dim Y_r =0$), the statement follows applying the arguments above to the pruning $X_{1,r}$ (resp.\ $X_{0,r-1}$, resp.\ $X_{1,r-1}$).
\end{proof}

The line bundle $\cA\in\Pic(\CX)$ is ample by \cite[Lemma~6.4]{WORS6}. Furthermore, applying the same lemma to the prunings $X_{i,j}$ of $X$, we get the following. 

\begin{lemma}\label{lem:nefCX1}
For every $i<j$, $i,j\in\{0,\dots, r\}$, the natural morphism $\CX\to \CX_{i,j}$ is supported by the nef line bundle $\bigotimes_{i\leq k\leq j}\cL_{k,k}$.
\end{lemma} 

\begin{remark}\label{rem:nefCX1}
In particular, the maps $s,d$ in
$$
\xymatrix@=6mm{&\CX\ar[ld]_{s}\ar[rd]^{d}&\\\CX_{0,r-1}&&\CX_{1,r}}
$$
are supported by the nef divisors $\cD_r$, $\cD_0$ and contract the curves $\Gamma_r$, $\Gamma_0$, respectively. For this reason, we will denote these maps by $\pi_r$ and $\pi_0$, respectively. Note that the above statement together with \cite[Lemma~2.8]{WORS1} imply that the following conditions are equivalent:
\begin{itemize}
\item $\dim Y_0=0$. 
\item $\cL_{0,0}$ is trivial.
\item $\pi_0\colon \CX\to \CX_{1,r}$ is an isomorphism. 
\item $\nu^-_1=1$.
\end{itemize} 
In the case in which $\dim Y_0 >0$, the exceptional locus of $\pi_0$ is $E_{1}$, and the fibers of $(\pi_0)_{|E_{1}}\colon E_{1}\to \pi_0(E_{1})$ are projective spaces of dimension $\nu^+_1-1$ (see the proof of \cite[Lemma~5.8]{WORS6}). An analogous equivalence holds for $\pi_r\colon \CX\to \CX_{0,r-1}$. 

In  particular, if $\dim Y_0 >0$ (resp.\ if $\dim Y_r >0$), $\Nef(\CX)$ has a facet corresponding to the elementary contraction supported by $\cD_r$ (resp.\ $\cD_0$). We will now show that the remaining facets of $\Nef(\CX)$ are determined by the nef divisors $\cD_i$, $i\in\{1,\dots,r-1\}$, introduced above. Note that $\cD_i$ is the supporting divisor of the induced morphism
$$\pi_i\colon \CX \lra \prod_{j\not =  i} \GX_{j,j}.$$ 
\end{remark} 

\begin{proposition}\label{prop:nefCX2}
For every $i\in\{1,\dots, r-1\}$, the line bundle $\cD_i$ is nef and not ample. Moreover, the associated morphism $\pi_i\colon \CX \to \prod_{j\not =  i} \GX_{j,j}$  contracts the primal curves of type $\Gamma_i$ to a point, and its exceptional locus is contained in $E_{i-1}\cap E_{i+1}$ $($setting  $E_0=E_r=\CX)$.
\end{proposition}

\begin{proof}
We will show that for every $i$, the induced map $\pi_i\colon \CX \to \prod_{j\not =  i} \GX_{j,j}$ is not finite. Two cycles $[C],[C'] \in \CX$ have the same image via $\pi_i$ if and only if their projections to $\CX_{0,i-1}$ and $\CX_{i+1,r}$ are the same, that is, if the following three conditions hold:
\begin{itemize}
\item Their fixed points of weights different from $i$ coincide. 
\item Their irreducible components  with sink in $Y_j$, $j < i-1$, are equal.
\item Their irreducible components with source in $Y_j$, $j > i+1$, are equal.
\end{itemize}

In particular, if $[C]\neq [C']$ have the same image via $\pi_i$, then they are singular at points of  $Y_{i-1}$ and $Y_{i+1}$, have the same fixed points in those components (say $P_{i-1}$, $P_{i+1}$), their subcycles contained in $B_{i-1}^+$ and $B_{i+1}^-$ are equal, while their subcycles joining $P_{i-1}$ and $P_{i+1}$ are different. This, on one hand, tells us that the exceptional locus of $\pi_i$ is contained in $E_{i-1}\cap E_{i+1}$, and on the other that $\pi_i$ contracts the primal curve $\Gamma_i$. This concludes the proof.
\end{proof}

\begin{remark}
By Proposition~\ref{prop:nefCX2}, the contractions $\pi_i$ are small if $2 \le i \le r-2$. 
Moreover, we already know that $\pi_i$ is a smooth blowup if $i$ is the minimum or maximum index belonging to $I$.
On the other hand, in the case in which $0\in I$ (resp.\ $r \in I$), the contraction $\pi_1$ (resp.\ $\pi_{r-1}$) can be small, divisorial, or of fiber type, as shown in the following examples (see also Remark~\ref{rem:nefcone} below).
\end{remark}

\begin{example}[Fiber type]\label{ex:ft} 
Consider the projective space $\DA_{5}(1)= \PP^{5}$, polarized with the hyperplane line bundle, endowed with the equalized bandwidth $1$ action whose extremal fixed-point components are two disjoint linear subspaces $\Lambda_-$ and $\Lambda_+$ of dimension $2$. We consider the induced action on \[X:=\DA_5(2) = \G(1,5).\]
The action on $\DA_{5}(2)$ has bandwidth $2$ with respect to the Pl{\"u}cker line bundle, and three fixed-point components; the extremal ones, $Y_0$ and $Y_2$, are the subvarieties parametrizing lines contained in $\Lambda_-$ or $\Lambda_+$, both isomorphic to $\P^2$. 
The Chow quotient of $X$ has dimension $7$; hence the morphism $\pi_1\colon \CX\to Y_0\times Y_2$ is necessarily of fiber type.
\end{example}

\begin{example}[Divisorial]\label{ex:div}
  Let us consider the odd-dimensional quadric $\DB_{5}(1)=\mathbb Q^{9} \subset \PP^{10}$, polarized with the restriction of hyperplane line bundle, endowed with the equalized action that has isolated sink and source (see \cite[Section~3.2]{FS}) and the induced action of the corresponding ($14$-dimensional) orthogonal Grassmannian of lines in the quadric
\[X=\DB_5(2).\]

The action on $\DB_{5}(1)$ has isolated extremal fixed points $z_-,z_+$ and a unique inner fixed-point component~$Z_0$, isomorphic to a smooth quadric $\DB_{4}(1)$, which is the intersection of $\DB_5(1)$ with the projective tangent spaces at $z_-$ and $z_+$. The action on $\DB_{5}(2)$ has three fixed-point components: The extremal ones are the subvarieties parametrizing lines of $\DB_5(1)$ passing through $z_-$ or $z_+$; they are both isomorphic to $\DB_{4}(1)=\mathbb Q^{7}$. The inner one is the subvariety parametrizing lines contained in $Z_0$, which is isomorphic to $\DB_{4}(2)$.

Given two points $y_0 \in Y_0$, $y_2 \in Y_2$ corresponding to two disjoint lines $\ell_0,\ell_2$, the linear span of $\ell_0$ and $\ell_2$ cuts $\DB_{5}(1)$ in a smooth quadric surface. The conic in $\DB_5(2)$ parametrizing the lines in the ruling containing $\ell_0$ and $\ell_2$ is the closure of an orbit with sink $y_0$ and source $y_2$, and it defines an element of $\CX$, which is the unique cycle passing through $y_0$ and $y_2$. In particular, this shows that the morphism $\CX \to Y_0 \times Y_2$, which assigns to every invariant cycle its intersection with $Y_0$ and $Y_2$ is surjective, hence, computing dimensions, birational.

On the other hand, given $y_0 \in Y_0$, there exist a unique point $y_0^*\in Y_2$ such that the lines $\ell_0,\ell_0^*$ parametrized by these two points meet, at a point $z_0 \in Z_0$. Given a general $3$-dimensional linear space $\Lambda$ containing $\ell_0$ and $\ell_0^*$ and contained in the projective tangent space $\mathbb T_{z_0}\mathbb Q^{9}$, its intersection with $\mathbb Q^{9}$ is a quadric cone containing $\ell_0$ and $\ell_0^*$. The lines in this cone are parametrized by a smooth conic in $\DB_{5}(2)$, which is the closure of an orbit with sink $y_0$ and source $y_0^*$. An easy computation shows that there is a $6$-dimensional family of such curves.

Hence the morphism $\pi_1\colon \CX \to Y_0 \times Y_2$ has positive-dimensional fibers over the subvariety $(y_0,y_0^*)$, which has dimension $7 = \dim Y_0=\dim Y_2$, and these fibers have  dimension $6$. We conclude that $\pi_1$ has exceptional locus of dimension $13= \dim \CX -1$, and so it is divisorial. 
\end{example}

\begin{example}[Small]\label{ex:sm}  Let us start from the $\C^*$-action on the quadric $\DB_5(1)$ of Example~\ref{ex:div} and now consider the induced action on the ($18$-dimensional) orthogonal Grassmannian \[X:=\DB_5(3)\] of planes in the quadric. It has criticality $2$ and three fixed-point components 
  \[
  Y_0\simeq\DB_4(2),\quad Y_1\simeq\DB_4(3),\quad Y_2\simeq\DB_4(2).
  \]
In fact, $Y_0$ and $Y_2$ are the subvarieties parametrizing planes of $\DB_5(2)$ passing through $z_-$ or $z_+$, so they are both isomorphic to $\DB_{4}(2)$. The inner one is the subvariety parametrizing planes contained in $Z_0$, which is isomorphic to $\DB_{4}(3)$.

We claim that the contraction $\pi_1\colon \CX\to Y_0\times Y_2$ is small. In fact, we first note that a point in $Y_0\times Y_2$ corresponds to a pair of lines $\Lambda_0,\Lambda_2$ in $\DB_4(1)$. With arguments similar to the ones used in Example~\ref{ex:div}, we can show that the image of $\pi_1$ corresponds to the set of pairs $(\Lambda_0,\Lambda_2)$ such that $\Lambda_0\cap \Lambda_2\neq \emptyset$, and the image of the exceptional locus corresponds to the set $\{(\Lambda_0,\Lambda_2)|\,\, \Lambda_0=\Lambda_2\}\simeq \DB_4(2)$, which has dimension~$11$. The inverse image of each of these elements is $4$-dimensional; hence the exceptional locus of $\pi_1$ has dimension~$15$, \textit{i.e.}, codimension~$2$ in $\CX$, which has dimension $17 = \dim \DB_5(3)-1$.
\end{example}

Let us conclude this section by  describing the nef and  Mori cones of $\CX$. 

\begin{theorem}\label{thm:nefcone}
The intersection number $\cL_{i,i}\cdot\Gamma_j$ is positive if\, $i=j$ and zero otherwise. In particular, the numerical class of the curves $\Gamma_j$ is well defined for every $j$. Moreover, the nef cone and the Mori cone of\, $\CX$ are simplicial, generated by the numerical classes of the divisors $\cL_{i,i}$, $i\in I$, and of the curves $\Gamma_i$, $i\in I$, respectively.
\end{theorem}

\begin{proof} By Proposition~\ref{prop:nefCX2},  $\cD_i \cdot \Gamma_i=0$ for every $i \in I$; in particular---since every $\cL_{j,j}$ is nef---$\cL_{j,j}\cdot \Gamma_j=0$ for $j\neq i$. Since, by Lemma~\ref{lem:intnumbers1}, $E_j \cdot \Gamma_j=2$, we see that the numerical class of the curves $\Gamma_j$ constructed in Construction~\ref{con:gij} is well defined. Abusing notation, we will write this class as $\Gamma_j\in\Nu(\CX)$.

By \cite[Lemma 6.4]{WORS6}, $\cA=\cD_i \otimes \cL_{i,i}$ is ample on $\CX$, so $\cL_{i,i} \cdot \Gamma_i >0$. 
The set $\{\cL_{i,i}\}_{i \in I}$ is linearly independent in $\NU(\CX)$ by Lemma~\ref{lem:pic}. Therefore, the set $\{\Gamma_i\}_{i \in I}\subset \Nu(\CX)$ is also linearly independent, so every class in $\overline{\NE(\CX)}$ can be written as a real linear combination 
\[
\sum a_i\Gamma_i.
\]
Intersecting with $\cL_{i,i}$, we get that $a_i \ge 0$ for every $i \in I$, so $\NE(\CX)$ is the cone generated by  $\{\Gamma_i\}_{i \in I}$, and by duality,  we get the statement on the nef cone.
\end{proof}

\begin{remark}\label{rem:nefcone}
Let us consider the case in which $\dim Y_0$, $\dim Y_r >0$. As previously noted, the contractions $\pi_0,\pi_r$ are always divisorial, and the contractions $\pi_i$, $i=2,\dots,r-2$, are always small. In the case $r=2$, we have seen examples (Examples~\ref{ex:ft},~\ref{ex:div},~\ref{ex:sm}) in which $\pi_1$ is of fiber type, divisorial, and small. When $r>2$, the contractions $\pi_1,\pi_{r-1}$ are necessarily birational (in fact, the birational maps $\CX\to \GX_{2,2}$, $\CX\to \GX_{r-2,r-2}$ factor via $\pi_1,\pi_{r-1}$, respectively). We have represented in Figure~\ref{fig:cones} the Mori cones of these Chow quotients in the cases $r=2,3,4$. 

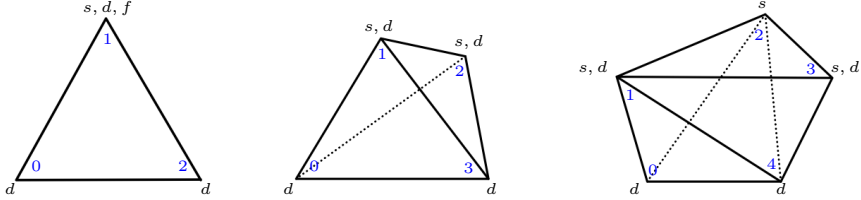
\begin{figure}[ht!!]
\begin{tikzpicture}
\node at (0,0) {\includegraphics[width=11cm]{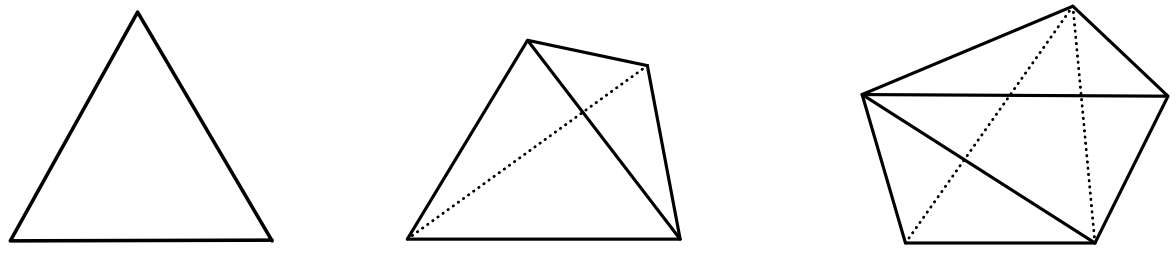}};
\node at (-5.5,-1.15) {\tiny $d$};
\node at (-2.9,-1.15) {\tiny $d$};
\node at (-4.2,1.25) {\tiny $s,d,f$};
\node at (-5.15,-0.85) {\color{blue}\tiny $0$};
\node at (-3.2,-0.85) {\color{blue}\tiny $2$};
\node at (-4.2,0.85) {\color{blue}\tiny $1$};
\node at (-1.8,-1.15) {\tiny $d$};
\node at (0.9,-1.15) {\tiny $d$};
\node at (-0.6,1) {\tiny $s,d$};
\node at (0.6,0.8) {\tiny $s,d$};
\node at (-1.45,-0.85) {\color{blue}\tiny $0$};
\node at (0.60,-0.87) {\color{blue}\tiny $3$};
\node at (-0.55,0.65) {\color{blue}\tiny $1$};
\node at (0.48,0.42) {\color{blue}\tiny $2$};
\node at (2.8,-1.15) {\tiny $d$};
\node at (4.75,-1.15) {\tiny $d$};
\node at (2.25,0.45) {\tiny $s,d$};
\node at (5.6,0.45) {\tiny $s,d$};
\node at (4.5,1.3) {\tiny $s$};
\node at (3.05,-0.9) {\color{blue}\tiny $0$};
\node at (4.63,-0.82) {\color{blue}\tiny $4$};
\node at (2.75,0.10) {\color{blue}\tiny $1$};
\node at (5.15,0.45) {\color{blue}\tiny $3$};
\node at (4.46,0.90) {\color{blue}\tiny $2$};
\end{tikzpicture}
\caption{The Mori cone of $\CX$ and the possible types of contractions ($d$=divisorial, $f$=fiber type, $s$=small) in the cases $r=2,3,4$, with no isolated extremal fixed points.\label{fig:cones}}
\end{figure}

The type (divisorial or small) of the birational contraction $\pi_1$ (resp.\ $\pi_{r-1}$) turns out to be determined by the induced action on the varieties $\cU_x$, $x\in Y_1$ (respectively, $x \in Y_{r-1}$), which has bandwidth~$2$. If there is an orbit linking two given elements of the sink and the source of $\cU_x$, then the map $\pi_1$ (resp.\ $\pi_{r-1}$) is divisorial; if that is not the case, then the map is small. A more detailed study of the (nef, Mori, movable, effective) cones of $\CX$ and of the partial Chow quotients, with special attention to the case in which $X$ is rational homogeneous, will be the goal of a forthcoming paper. 
\end{remark}


\section{Anticanonical bundle of the Chow quotient}\label{sec:antican}

\begin{notation}\label{not:dualbasis}
Theorem~\ref{thm:nefcone} allows us to consider the basis $\{L_i\}_{i\in I}$ of $\NU(\CX)$ dual to $\{\Gamma_i\}_{i\in I}$, whose elements are determined by some effective $\Q$-divisors, which we denote by $L_i$ as well. 
\end{notation}

We now compute the class of the canonical bundle of $\CX$ as a linear combination of the $L_i$. Let us first introduce some appropriate invariants of the $\C^*$-action, which will appear in the coefficients of the final formula.

\begin{definition}\label{def:us}
Let us consider the ranks $\nu^\pm_i=\rk(N^\pm(Y_i))$ introduced in Remark~\ref{rem:convex}. We set 
\begin{equation}\label{eq:ui}
u_j^+:= \nu^-_j-\nu^-_{j+1}-1, \quad
u_j^-:= \nu^+_j-\nu^+_{j-1}-1.
\end{equation}
The values $u_j^\pm$ are the dimensions of the extremal fixed-point components of the induced action on $\cU_y$, for a general $y \in Y_j$ (see \cite[Corollary 5.14]{OS2}). We will not need this fact here, but we note that, by \cite[Corollary 5.14]{OS2}, we also have $u_j^+ + u_{j+1}^-  = \dim \cU_y$  for every $j= 0, \dots, r-1$. Since the dimension of $\cU_y$ is constant, we get 
\begin{equation}\label{eq:u-+2}
 u_{j-1}^+-u_j^+ =  u_{j+1}^--u_j^-,\quad j= 1, \dots, r-1.
\end{equation}
\end{definition}

Note that in the following statement, we do not need to distinguish the case of extremal isolated fixed points, since, if $\dim Y_0 = 0$ (resp.\ $\dim Y_r=0$), the coefficient of $L_0$ (resp.\ of $L_r$) appearing in the formulae is zero. 

\begin{theorem}\label{thm:can} Let $X$ be as in Setup~\ref{set:main}. If  
the criticality of the action on $X$ is equal to $2$, then
\[
-K_{\CX}= u_1^-L_0+(u_0^+-u_1^++2)L_1 +u_1^+L_2.
\]
If else the criticality of the action is greater than or equal to $3$, then 
\[
-K_{\CX} =u_1^-L_0 +\left(u_{0}^+ - u_1^+ +1\right)L_1 + \sum_{i=2}^{r-2} \left( u_{i-1}^+-u_i^+\right)L_i +\left(u_{r-2}^+ - u_{r-1}^+ +1\right)L_{r-1}
+u_{r-1}^+L_r.
\]
\end{theorem}

\begin{proof}
We will use  the description of $\CX$ as the result of a sequence of smooth blowups along the sides of the triangle (\ref{fig:Chow0}). The centers of these blowups have been described in \cite[Remark~5.10]{WORS6}: The center of the blowup $\CX_{0,i}\leftarrow \CX_{0,i+1}$ is the set of invariant cycles in $X_{0,i}$ whose strict transforms in $X$ have source in $Y_{i}$; it has codimension $\nu_i^-$. By \cite[Theorem 5.7]{OS2}, there exists a smooth $\C^*$-invariant cycle linking $Y_0$ and $Y_{i+1}$, so the center of $\CX_{0,i}\leftarrow \CX_{0,i+1}$ is not contained in the union of the exceptional divisors of $\GX_{0,1}\leftarrow \CX_{0,i}$. In particular, the inverse images of the exceptional divisors of $\GX_{0,1}\leftarrow \CX_{0,i}$ coincide with their strict transforms. Since  $-K_{\CX_{0,i+1}}+K_{\CX_{0,i}}$ is linearly equivalent to the exceptional divisor counted with multiplicity $(1-\nu_i^-)$ for every $i$, we conclude that    
\[
-K_{\CX}= -K_{\GX_{0,1}}+ \sum_{i=1}^{r-1}\left(1-\nu_i^-\right) E_i.
\]
Similarly, we also get that 
\[
-K_{\CX}= -K_{\GX_{r-1,r}}+ \sum_{i=1}^{r-1}\left(1-\nu_i^+\right) E_i.
\]

From the first expression, taking into account that for every $j\geq 2$ the curves $\Gamma_j$ are contracted by the natural map from $\CX$ to $\GX_{0,1}$, the intersection numbers computed in Proposition~\ref{prop:intnumbers}, and the definition of the invariants $u_i^{\pm}$, we have 
\[
-K_{\CX}\cdot \Gamma_j=
\left\{\def\arraystretch{1.3}
\begin{array}{ll}
-\left(1-\nu^-_{j-1}\right)+2\left(1-\nu^-_{j}\right)-\left(1-\nu^-_{j+1}\right)=u^+_{j-1}-u^+_{j}&\mbox{if }j<r-1,\\
-\left(1-\nu^-_{r-2}\right)+2\left(1-\nu^-_{r-1}\right)=u^+_{r-2}-u^+_{r-1}+1&\mbox{if }j=r-1,\\
-\left(1-\nu^-_{r-1}\right)=u^+_{r-1}&\mbox{if }j=r.
\end{array}
\right.
\]
Analogously, we may use the second expression of $-K_{\CX}$ to get 
\[-K_{\CX}\cdot\Gamma_0=u_1^-,\quad -K_{\CX}\cdot\Gamma_1=u^-_2-u_1^-+1=u_{0}^+ - u_1^+ +1.\]
The last equality follows from Formula (\ref{eq:u-+2}). 

The above intersection numbers prove the statement in the case $r\geq 3$. In the case $r=2$, these arguments provide only the anticanonical degrees of $\Gamma_0$ and $\Gamma_2$ (respectively, $u^-_1$, $u^+_1$), and a different one is needed to prove that $-K_{\CX}\cdot\Gamma_1=u_0^+-u_1^++2$. Since we know that $E_1\cdot\Gamma_1=2$ and $-K_{\CX}=-K_{\GX_{0,1}}+(1-\nu_1^-)E_1$, we need to prove that $-K_{\GX_{0,1}}\cdot \Gamma_1=\nu^-_0$. 

The variety $\GX_{0,1}$ is a $\P^{\nu^-_0-1}$-bundle over $Y_0$, and $\Gamma_1$ is clearly contracted by the composition $\CX\to \GX_{0,1}\to Y_0$, so it is enough to show that the image of $\Gamma_1$ in  $\GX_{0,1}$ is a line in a fiber of $\GX_{0,1}\to Y_0$. 

By construction, the points of $\Gamma_1$  are $\C^*$-invariant cycles that are the images in $X$ of a $1$-dimensional family of cycles in $\FF_2$ passing through two points $x_-,x_+\in\FF_2$; we denote the image of $x_-$ in $X$ by $P_0\in Y_0$. Since the $\C^*$-action is equalized, these cycles have different tangent directions at $P_0$, so the differential of the map $\FF_2\to X$ at $x_-$ is injective, and the projectivization of its image is a line in $\P(N^-(Y_0)_{P_0})$. This is precisely the image of $\Gamma_1$ in $\GX_{0,1}$, and the proof is finished.
\end{proof}

\begin{remark}\label{rem:howtouse}The formulae in Theorem~\ref{thm:can} can be used to compute the canonical bundle of the Chow quotient of a given equalized $\C^*$-action on a variety $X$ if the Fano index $i(X)$ of $X$ and the fixed-point components of the action are known. In fact, for every inner fixed-point component $Y_j$, we have, by definition and   by Corollary~\ref{cor:can},  
\[
\begin{cases}
\nu_j^+ + \nu_j^- = \codim (Y_j,X),\\
\nu_j^+ - \nu_j^- = i(X)j - \codim(Y_0,X).
\end{cases}
\]
We can then compute $\nu_j^\pm$ for every fixed-point component and, subsequently, $u_j^\pm$ using Formula (\ref{eq:ui}).
\end{remark}

\begin{example} The Chow quotients of the varieties in Examples~\ref{ex:ft},~\ref{ex:div}, and~\ref{ex:sm} are all Fano manifolds. In fact, we can compute their anticanonical bundles as explained above, obtaining
\[\begin{aligned}
 -K_{\CX}&= L_1 + 4L_2+L_3  &\mbox{when } X &=\DA_5(2), \nonumber \\
 -K_{\CX}&= L_1 + 6L_2+L_3  & \mbox{when } X &=\DB_5(2), \nonumber \\
 -K_{\CX}&= 2L_1 + 3L_2+2L_3 &\mbox{when } X &=\DB_5(3). \nonumber 
\end{aligned}
\]
In all cases, $-K_{\CX}$ is ample, by Kleiman's criterion.
\end{example}

\begin{example} The action on $\DB_6(5)$ induced by the equalized action on $\DB_6(1)$ with isolated sink and source (see \cite[Section~3.2]{FS}) has criticality $2$ and fixed-point components 
\[Y_0\simeq\DB_5(4),\quad Y_1\simeq\DB_5(5),\quad Y_2\simeq\DB_5(4).\]
We can compute that 
\[-K_{\CX}=4L_0 -L_1+4L_2,\]
so in this case $-K_{\CX}$ is not nef. Analogously, one may compute that the anticanonical bundle of the Chow quotient of $\DB_5(4)$ with respect to the action induced by the action on $\DB_5(1)$ with isolated sink and source  is $3L_0 +3L_2$, which is nef but not ample.
\end{example}


\newcommand{\etalchar}[1]{$^{#1}$}


\begin{thebibliography}{ORSC\etalchar{+}23b+++}

\bibitem[AO02]{AO2}
M.~Andreatta and G.~Occhetta, 
\emph{Special rays in the Mori cone of a projective variety}, 
Nagoya Math.~J.\ \textbf{168} (2002), 127--137,
\doi{10.1017/S0027763000008400}.

\bibitem[BB73]{BB}
A.~Bia{\l}ynicki-Birula, 
\emph{Some theorems on actions of algebraic groups},
Ann.\ of Math.~(2) \textbf{98} (1973), 480--497,
\doi{10.2307/1970915}.

\bibitem[BBCM02]{BBC}
A.~Bia{\l}ynicki-Birula, J.\,B.~Carrell, and W.\,M.~McGovern, 
\emph{Algebraic quotients. Torus actions and cohomology. The
adjoint representation and the adjoint action}, Encyclopaedia Math.\ Sci.,
vol.~131, 
Springer-Verlag, Berlin, 2002,
\doi{10.1007/978-3-662-05071-2}.

\bibitem[CP91]{CP}
F.~Campana and T.~Peternell, 
\emph{Projective manifolds whose tangent bundles are numerically effective}, 
Math.\ Ann.\ \textbf{289} (1991), no.~1, 169--187,
\doi{10.1007/BF01446566}.

\bibitem[Car02]{CARRELL}
J.\,B.~Carrell, 
\emph{Torus actions and cohomology}, 
in: \emph{Algebraic quotients. Torus actions and cohomology. The
  adjoint representation and the adjoint action}, pp.~83--158,
Encyclopaedia Math.\ Sci., vol.~131, Springer, Berlin, 2002,
\doi{10.1007/978-3-662-05071-2\_2}.

\bibitem[Dol03]{DolgachevGIT}
I.~Dolgachev, 
\emph{Lectures on invariant theory},
London Math.\ Soc.\ Lecture Ser., 
vol~296, 
Cambridge Univ.\ Press, Cambridge, 2003,
\doi{10.1017/CBO9780511615436}.

\bibitem[FSC24]{FS}
A.~Franceschini and L.\,E.~Sol\'{a}~Conde, 
\emph{Inversion maps and torus actions on rational homogeneous va\-ri\-e\-ties}, Geom.\ Dedicata \textbf{218} (2024), no.~1,
Paper No.~21, \href{https://doi.org/10.1007/s10711-023-00866-z}{\texttt{doi:10.1007/s10711-023-\allowbreak 00866-z}}

\bibitem[HK00]{HuKeel}
Y.~Hu and S.~Keel,
\emph{Mori dream spaces and GIT} (dedicated to William Fulton on the occasion of his 60th birthday),
Michigan Math.~J.\ \textbf{48} (2000), 331--348,
\doi{10.1307/mmj/1030132722}.

\bibitem[Ive72]{IVERSEN}
B.~Iversen, 
\emph{A fixed point formula for action of tori on algebraic varieties},
Invent.\ Math.\ \textbf{16} (1972), no.~3, 229--236,
\doi{10.1007/BF01425495}.

\bibitem[Kan17]{Kan5}
A.~Kanemitsu,
\emph{Fano $5$-folds with nef tangent bundles}, 
Math.\ Res.\ Lett.\ \textbf{24} (2017), no.~5, 1453--1475,
\doi{10.4310/MRL.2017.v24.n5.a6}.

\bibitem[Kap93]{Kap}
M.\,M.~Kapranov, 
\emph{Chow quotients of Grassmannians. {I}},
in: \emph{I.\,{M}.~{G}el'fand {S}eminar},  pp.~29--110,  Adv.\ Soviet
Math., vol.~16, Part~2,
Amer.\ Math.\ Soc., Providence, RI, 1993,
\doi{10.1090/advsov/016.2/02}.

\bibitem[Kol96]{kollar}
J.~Koll{\'a}r, 
\emph{Rational curves on algebraic varieties},
Ergeb.\ Math.\ Grenzgeb.\ (3) \textbf{32}, Springer-Verlag, Berlin, 1996, 
  \doi{10.1007/978-3-662-03276-3}. 

\bibitem[Lak82]{Lak82}
D.~Laksov, 
\emph{Notes on the evolution of complete correlations}, 
in: \emph{Enumerative geometry and classical algebraic geometry}
  ({N}ice, 1981), pp.~107--132, Progr.\ Math., vol.~24, 
Birkh\"auser Boston, Boston, MA, 1982,
\doi{10.1007/BF01585097}.

\bibitem[Lak87]{Lak87}
\bysame,
\emph{Completed quadrics and linear maps}, 
 in: \emph{Algebraic geometry, {B}owdoin, 1985} ({B}runswick, {M}aine,
  1985), pp.~371--387, Proc.\ Sympos.\ Pure Math., vol.~46, Part 2,
  Amer. Math. Soc., Providence, RI, 1987,
  \doi{10.1090/pspum/046.2/927988}.

\bibitem[Lak88]{Lak88}
\bysame, 
\emph{The geometry of complete linear maps}, 
Ark.\ Mat.\  \textbf{26} (1988), no.~2, 231--263,
\doi{10.1007/BF02386122}.

\bibitem[Mas20a]{Mas1}
A.~Massarenti, 
\emph{On the birational geometry of spaces of complete forms {I}:
  collineations and quadrics}, 
Proc.\ Lond.\ Math.\ Soc.~(3) \textbf{121} (2020), no.~6, 1579--1618,
\doi{10.1112/plms.12377}.

\bibitem[Mas20b]{Mas2}
\bysame,
\emph{On the birational geometry of spaces of complete forms {II}:
  {S}kew-forms}, 
J.~Algebra \textbf{546} (2020), 178--200,
\doi{10.1016/j.jalgebra.2019.10.047}.

\bibitem[MFK94]{MFK}
D.~Mumford, J.~Fogarty, and F.~Kirwan, 
\emph{Geometric invariant theory}, 3rd ed., Ergeb.\ Math.\ Grenzgeb.~(2), 
vol.~34,  Springer-Verlag, Berlin, 1994,
\doi{10.1007/978-3-642-57916-5}.

\bibitem[MOSW\etalchar{+}15]{MOSWW}
R.~Mu{\~n}oz, G.~Occhetta, L.\,~E.~Sol{\'a}~Conde, K.~Watanabe,
  and J.\,A.~Wi\'sniewski, 
\emph{A survey on the Campana-Peternell conjecture}, 
Rend.\ Istit.\ Mat.\ Univ.\ Trieste, \textbf{47} (2015), 127--185,
\doi{10.13137/0049-4704/11223}.

\bibitem[ORSW22]{WORS3}
 G.~Occhetta, E.\,A.~Romano, L.\,E.~Sol\'a Conde, and J.\,A.~Wi\'sniewski, 
\emph{Small modifications of Mori dream spaces arising from
  {$\mathbb{C}^*$}-actions}, 
Eur.~J.\ Math.\ \textbf{8} (2022), no.~3, 1072--1104,
\doi{10.1007/s40879-022-00540-w}.



\bibitem[ORSW23a]{WORS5}
\bysame, 
\emph{Rational homogeneous spaces as geometric realizations of birational
  transformations}, 
Rend.\ Circ.\ Mat.\ Palermo~(2) \textbf{72} (2023), no.~6, 3223--3253, 
\doi{10.1007/s12215-023-00880-w}.

\bibitem[ORSW23b]{WORS1}
\bysame, 
\emph{Small bandwidth {$\mathbb C^*$}-actions and birational geometry}, 
J.~Algebraic Geom.\ \textbf{32} (2023), no.~1, 1--57,
\doi{10.1090/jag/808}.

\bibitem[ORSW25]{WORS6}
\bysame, 
\emph{Chow quotients of $\mathbb{C}^*$-actions}, 
Int. Math. Res. Not. IMRN. \textbf{12} (2025), 
\doi{https://doi.org/10.1093/imrn/rnaf154}

\bibitem[OS26]{OS2}
G.~Occhetta and L.\,E.~Sol{\'a}~Conde, 
\emph{Characterizing rational homogeneous spaces via {$\mathbb
  C^*$}-actions}, 
Collect.\ Math.\ \textbf{77} (2026), no.~1, 161--193,
\doi{10.1007/s13348-024-00463-7}.

\bibitem[RW22]{RW}
E.\,A.~Romano and J.\,A.~Wi\'{s}niewski, 
\emph{Adjunction for varieties with a {$\mathbb{C}^*$} action}, 
Transform.\ Groups \textbf{27} (2022), no.~4, 1431--1473,
\doi{10.1007/s00031-020-09627-8}.

\bibitem[Ser06]{Ser1}
E.~Sernesi, 
\emph{Deformations of algebraic schemes},
Grundlehren math.\ Wiss.,  vol.~334, 
Springer-Verlag, Berlin, 2006.

\bibitem[Tal09]{Talpo}
M.~Talpo, 
\emph{Deformation theory}, 
Master's thesis, University of Pisa, 2009, available at \url{https://people.dm.unipi.it/talpo/stuff/tesi.pdf}.

\bibitem[Tha96]{Thaddeus1996}
M.~Thaddeus, 
\emph{Geometric invariant theory and flips}, 
J.~Amer.\ Math.\ Soc.\ \textbf{9} (1996), no.~3, 691--723,
\doi{10.1090/S0894-0347-96-00204-4}.

\bibitem[Tha99]{Thaddeus}
\bysame, 
\emph{Complete collineations revisited}, 
Math.\ Ann.\ \textbf{315} (1999), no.~3, 469--495,
\doi{10.1007/s002080050324}.

\bibitem[Vai82]{Vai82}
I.~Vainsencher, 
\emph{Schubert calculus for complete quadrics}, 
in: \emph{Enumerative geometry and classical algebraic geometry}
  ({N}ice, 1981),  pp.~199--235, Progr.\ Math., vol.~24, 
Birkh\"auser, Boston, MA, 1982,
\doi{10.1007/978-1-4684-6726-0_10}.

\bibitem[Vai84]{Vai84}
  \bysame,
  \emph{Complete collineations and blowing up determinantal ideals}, 
  Math.\ Ann.\ \textbf{267} (1984), no.~3, 417--432,
  \doi{10.1007/BF01456098}.

\end{thebibliography}
\end{document}